\documentclass[11pt]{article}
\usepackage{amssymb, amsmath}

\usepackage[margin=1.2in]{geometry}

\usepackage{mathptmx}
\usepackage{layout}
\usepackage{theorem}

\newtheorem{theorem}{Theorem}
\newtheorem{lemma}{Lemma}

\newtheorem{corollary}{Corollary}
\newtheorem{conjecture}{Conjecture}
\theorembodyfont{\upshape}\newtheorem{remark}{Remark} 

\newenvironment{proof}[1][Proof]{\begin{trivlist}
\item[\hskip \labelsep {\bfseries #1}]}{\end{trivlist}}

\newcommand{\qed}{\nobreak \ifvmode \relax \else
      \ifdim\lastskip<1.5em \hskip-\lastskip
      \hskip1.5em plus0em minus0.5em \fi \nobreak
      \vrule height0.75em width0.5em depth0.25em\fi}

\renewcommand{\Re}{\operatorname{Re}}

\begin{document}

\title{A Theory of Intermittency Renormalization of Gaussian Multiplicative Chaos Measures}
\author{Dmitry Ostrovsky}

\date{}

\maketitle
\noindent

\begin{abstract}
\noindent
A theory of intermittency differentiation is developed for a general class of Gaussian Multiplicative Chaos measures
including the measure of Bacry and Muzy on the interval and circle as special cases. 
An exact, non-local functional equation is derived for the derivative of a general functional of the total mass of the measure with respect to
intermittency. 
The formal solution is given 
in the form of an intermittency expansion
and proved to be a renormalized expansion in the centered moments of the total mass of the measure. The full intermittency expansion
of the Mellin transform of the total mass is computed in terms of the corresponding expansion of log-moments.
The theory is shown to extend to the dependence structure of the measure. For application, the intermittency expansion of 
the Bacry-Muzy measure on the circle is computed exactly, and the Morris integral probability distribution is shown
to reproduce the moments of the total mass and the intermittency expansion, resulting in the conjecture that it is
the distribution of the total mass. It is conjectured in general that the intermittency expansion 
captures the distribution of the total mass uniquely. 
\end{abstract}

\section{Introduction}
\noindent The theory of Gaussian Multiplicative Chaos (GMC) measures has greatly advanced since its inception in 1972 by Mandelbrot \cite{secondface}, 
who introduced the key ingredients of what is now known as GMC under the name of the limit lognormal measure, cf. also his review \cite{Lan}.
The mathematical foundation of the subject was laid down by Kahane \cite{K2}, who created a comprehensive, mathematically rigorous theory of multiplicative 
chaos measures based on his theory of convergence of a particular class of positive martingales. The theory was advanced further
around 2000 with the introduction of the conical set construction by Barral and Mandelbrot \cite{Pulses} and Schmitt and Marsan \cite{Schmitt} and assumed its modern form with
the theory of infinitely divisible multiplicative chaos measures of Bacry and Muzy \cite{BM1}, \cite{BM} that is based on a spectral representation of
infinitely divisible processes of Rajput and Rosinski \cite{RajRos}. The theory of Bacry and Muzy was limited to multiplicative chaos on a finite interval.
It has since been extended  in the gaussian case to multiple dimensions by Robert and Vargas \cite{RV}, who also relaxed Kahane's $\sigma-$positivity 
condition and proved universality, to other geometric shapes such as the circle by Fyodorov and Bouchaud \cite{FyoBou} and Astala \emph{et. al} \cite{Jones}, as well as to critical multiplicative chaos by Duplantier \emph{et. al} \cite{dupluntieratal} and Barral \emph{et. al.} \cite{barraletal}, and most recently to super-critical multiplicative chaos by Madaule \emph{et. al.} \cite{Mad}. Most recently, Berestycki \cite{B}, Junnila and Saksman \cite{JS}, and Shamov \cite{Shamov} found new 
re-formulations and further extended the existing theory. 
In the general  infinitely divisible case the theory of Bacry and Muzy was further advanced by Rhodes and Vargas \cite{RV3} and Barral and Jin \cite{barralID},
and we derived key invariance properties of the underlying infinitely divisible field in \cite{Me5} and a formula for the moments of the total mass in \cite{MeLMP}. 

The interest in multiplicative chaos derives from its remarkable property of multifractality, from complexity of
mathematical problems that it poses such as understanding its stochastic dependence structure, and from the many applications in 
theoretical and statistical physics, in which it naturally appears.  Without aiming for comprehension, we can mention applications to
conformal field theory and quantum gravity \cite{BenSch}, \cite{DS}, \cite{RV1}, statistical mechanics of disordered energy landscapes and extrema of the 2D gaussian free field \cite{FyoBou}, \cite{FLDR},  \cite{FLDR2}, \cite{FyodSimm}, \cite{Madmax}, \cite{Me16}, a theory of conformal weldings \cite{Jones}, and even conjectured \cite{YK}, \cite{Menon} and some rigorous \cite{SW} applications to the behavior of the Riemann zeta function on the critical line. 

A fundamental open problem in the theory of GMC is to calculate the distribution of the total mass of the chaos measure and, more generally,
understand its stochastic dependence structure, \emph{i.e.} the joint distribution of the measure of several subsets of the set, on which it is defined. 
The contribution of this paper is to advance this problem 
for the class of multi-dimensional GMC measures introduced by Robert and Vargas \cite{RV}. 
These measures have the property that the positive integer moments of its total mass are known in the form of a multiple integral of Selberg type. 
For example, the moments of the total mass of the Bacry-Muzy GMC measure are given by the classical Selberg integral \cite{MRW} on the interval 
and by the Morris integral on the circle \cite{FyoBou}, \cite{Me16}. 
The primary challenge of 
recovering the distribution from the moments is that the moments become infinite at any level of intermittency (also referred to as the inverse temperature 
in the statistical physics literature). Hence, the problem of recovering the distribution from the moments is that of renormalization, \emph{i.e.} of removing infinity 
from the moments and re-summing them so as to reconstruct the distribution. In the special case of the Bacry-Muzy GMC measure on the interval we 
developed in a series of papers \cite{Me2}--\cite{MeIMRN}  the theory of intermittency differentiation that allowed us to compute the full high-temperature (low intermittency) expansion of the Mellin transform of the total mass and effectively reconstruct the Mellin transform by summing the intermittency expansion. We then
checked that the resulting expression is the Mellin transform of a valid probability distribution, known as the Selberg integral probability distribution, 
having the properties that its positive integer moments are given by the Selberg integral and that the asymptotic expansion of its Mellin transform coincides 
with the intermittency expansion. Thus, we constructed a good candidate for the distribution of the total mass in the sub-critical regime and then calculated 
the critical distribution as a simple limiting case, cf. \cite{Me14}, \cite{Me16}. 

It is worth emphasizing that the main achievement of the intermittency differentiation approach is that it provides an exact mechanism for renormalization.
It is possible in some cases to guess the Mellin transform, in the sense of a function of a complex variable whose restriction to the finite interval of positive integers, where the moments are finite, coincides with the moments, cf. \cite{CRS}, \cite{FyoBou}, \cite{FLDR}, \cite{FLD}, \cite{Me16}. While this method 
produces the same formulas for the Mellin transform of the total mass of the Bacry-Muzy GMC measure on the interval and circle as ours, 
it does not capture the distribution uniquely because it operates on the moments directly and the moment problem is not determinate. 

The contribution of this paper is threefold. 
Our main contribution is to extend the theory of intermittency renormalization to a general multi-dimensional GMC measure.
The GMC measure is defined as the exponential functional of a regularized gaussian process with logarithmic covariance in the limit of zero 
regularization. The theory that is developed in this paper is applicable to the exponential functional of a properly normalized singular gaussian process
in general and so is not limited to GMC measures \emph{per se}. However, logarithmically-correlated gaussian processes as considered by
Bacry-Muzy \cite{MRW} on the interval, Fyodorov and Bouchaud \cite{FyoBou} and Astala \emph{et. al.} \cite{Jones} and on the circle, and Robert and Vargas \cite{RV} on $\mathbb{R}^d$ provide the main example of such singular gaussian processes that have a well-developed theory of exponentiation
based on regularization. Hence, for practical purposes, we limit ourselves to GMC measures. We show that our theory of intermittency differentiation
and renormalization extends to general GMC measures. In particular, we prove that the intermittency expansion is an exactly renormalized expansion in the centered moments of the total mass and compute it explicitly in terms of derivatives of log-moments at zero intermittency thereby giving a formal solution to the problem of renormalization in the high temperature phase. 
As an application, we treat the case
of the Bacry-Muzy GMC on the circle and compute the Mellin transform of the total mass exactly, recovering the results of \cite{FyoBou} and \cite{Me16}
that were obtained heuristically. The second contribution is to give a comprehensive
exposition of our approach. Originally, the derivations of the intermittency differentiation rule, of intermittency expansions, and of the Mellin transform
were presented in different publications with varying degrees of generality. The goal of this paper is to collect them all in one place so as to
emphasize their generality and make our approach accessible to a wider audience. Finally, the third contribution is to compile a list of conjectures
that might lead to further advances in the future, the primary of which is that the intermittency expansion is convergent for a class of smooth test
functions and therefore captures the distribution of the total mass of the GMC measure uniquely. 

Our paper is limited to the problem of the distribution of the total mass of GMC measures and to the derivation of the intermittency expansion of its Mellin transform. We will not attempt to review the theory of the Selberg and Morris integral probability distributions, which are conjectured to be the distributions 
of the total mass of the Bacry-Muzy GMC measure on the interval and circle, respectively, as it would lead us to the subject of Barnes beta probability 
distributions \cite{Me13}, \cite{Me14} that is outside the scope of this paper and that we recently reviewed in \cite{Me16}. We will also not review the theory of multiplicative chaos measures \emph{per se} and refer the interested reader to \cite{RV2} for a general review in the gaussian case, to \cite{MeIMRN} and \cite{Menon} for detailed reviews of the gaussian case on the interval and to \cite{Me5} and \cite{MeLMP} for the infinitely divisible case on the interval. 

Our results and derivations are all exact in the sense of equality of formal power series but not mathematically rigorous with the exception
of results in Section 5. 

The plan of the paper is as follows. In Section 2 we briefly review the general GMC construction following the works of Bacry and Muzy 
and Robert and Vargas and state the formula for
the moments of the total mass. In Section 3 we state our main results: the intermittency differentiation rule, properties of the 
intermittency expansion, high temperature (low intermittency) expansion of the Mellin transform, and extensions of our approach to multiple subsets, \emph{i.e.} the dependence structure of the GMC measure.
In Section 4 we give formal derivations of the key results from the first principles. In Section 5
we calculate the distribution of the total mass of the Bacry-Muzy GMC measure on the circle and relate it to the Morris integral probability distribution. 
In Section 6 we list a number of key conjectures and open problems. Conclusions are given in Section 7. The two appendices present the general 1D GMC measure  on the interval as a deformation of the Bacry-Muzy construction and a second derivation of the differentiation rule, respectively.

\section{A Brief Review of Gaussian Multiplicative Chaos}
\noindent Consider a stationary  
gaussian process on $\mathbb{R}^d$ having the general logarithmic covariance of the form
\begin{align}
{\bf{Cov}}\left[\omega_{\mu,\varepsilon}(x), \,\omega_{\mu, \varepsilon}
(y)\right]  & = \mu \bigl(\theta_\varepsilon\star g\bigr)(x-y), \\
& = \mu \int_{\mathbb{R}^d} g(x-y-\varepsilon z)\,\theta(z)\,dz,
\end{align}
and the mean
\begin{equation}\label{mean}
{\bf E} \left[ \omega_{\mu,\varepsilon}(x) \right] = -\frac{1}{2} {\bf Var}  \left[ \omega_{\mu,\varepsilon}(x) \right],
\end{equation}
where $\theta(z)$ is a positive-definite bump function decaying sufficiently fast at infinity, $\theta_\varepsilon(z)=\theta(z/\varepsilon)/\varepsilon^d,$
\begin{equation}\label{bump}
\int_{\mathbb{R}^d}  \theta(z)\,dz =1,
\end{equation}
and $g(x)$ is a positive-definite function of the form
\begin{equation}\label{glog}
g(x) = \log^+ \frac{1}{|x|} + h(x).
\end{equation}
Here $h(x)$ is bounded and continuous and
\begin{equation}
\log^+ (z) \triangleq 
 \begin{cases}
 \log(z), \, z\geq 1, \\
 0, \, 0<z<1.
 \end{cases}
\end{equation}
The parameter $\mu>0$ is known as intermittency and is often written in the form $\mu=2\beta^2,$ in which case
$\beta$ is referred to as the inverse temperature. 
This construction is due to Robert and Vargas \cite{RV}. In the special case of $d=1,$ letting $g(x)=-\log(x),$ 
$h(x)=0,$ and $\theta(x)=1$ one recovers the 
Bacry-Muzy construction \cite{BM} on the unit interval. The primary example of the $g(x)$ function is afforded by a $d-$dimensional generalization of the
conical construction of Bacry-Muzy that is due to Chainais \cite{Chainais}. Let the $d-$dimensional cone be defined by
\begin{equation}
C(x) = \big\{(y,t)\in \mathbb{R}^d\times \mathbb{R}_+ \, \Big|\, |y-x|\leq \min(t, 1)/2\big\}.
\end{equation}
Then, the function $g(x)$, defined by
\begin{equation}
g(x)\triangleq \int_{C(x)\cap C(0)} \frac{dydt}{t^{d+1}},
\end{equation}
is positive-definite and satisfies Eq. \eqref{glog} for some bounded and continuous $h(x),$ cf. \cite{Chainais}.

Now, fix a ball $\mathcal{D}\subset\mathbb{R}^d$ and consider the associated random measure (also known as the partition function in the physics literature),
\begin{equation}\label{Meps}
M_{\mu, \varepsilon}[\varphi](\mathcal{D})\triangleq \int_\mathcal{D} \varphi(x)\,e^{ \omega_{\mu,\varepsilon}(x)} \, dx.
\end{equation}
We will assume for simplicity that $\varphi(x)>0.$ 
The measure is normalized so that
\begin{equation}
{\bf E}\bigl[M_{\mu, \varepsilon}[\varphi](\mathcal{D})\bigr] = \int_\mathcal{D} \varphi(x)\,dx,
\end{equation}
due to
\begin{equation}
{\bf E}\Bigl[e^{ \omega_{\mu,\varepsilon}(x)}\Bigr] = 1,
\end{equation}
which follows from the condition in Eq. \eqref{mean}.

The theories of Kahane \cite{K2} and Robert-Vargas \cite{RV} imply that the limit is a universal (independent of the choice of $\theta(z)$) non-trivial random measure $M_\mu(dx)$ for a range of $\mu\in[0, \mu_c)$ so that
\begin{equation}
\lim\limits_{\varepsilon\rightarrow 0} M_{\mu, \varepsilon}[\varphi](\mathcal{D})=\int_\mathcal{D} \varphi(x)\,M_\mu(dx).
\end{equation}
The positive integer moments of the total mass\footnote{By a slight abuse of terminology, we refer to any integral of the form $\int_\mathcal{D} \varphi(x)\,M_\mu(dx)$ as the total mass.}  can be calculated up to some critical value. Denote
\begin{equation}
S_n[\varphi](\mu) = {\bf E}\Bigl[\Bigl(\int_\mathcal{D} \varphi(x)\,M_\mu(dx)\Bigr)^n\Bigr].
\end{equation}
Then, the standard gaussian calculation shows that the $n$th moment is
\begin{equation}\label{momentsint}
S_n[\varphi](\mu) =  \int\limits_{\mathcal{D}^n} \prod_{i=1}^n \varphi(x_i)\,
\exp\Bigl(\mu \sum_{i<j}^n g(x_i-x_j)\Bigr) dx_1\cdots dx_n, \; n<2d/\mu,
\end{equation}
and is infinite otherwise. In fact, it is easy to see that the contribution of the region
where the integrand is large, \emph{i.e.} where the points are within
$\varepsilon$ apart, is of the order
\begin{equation}\label{regioncontrib}
O\bigl(\varepsilon^{d(n-1)} \varepsilon^{ - \mu n(n-1)/2}\bigr).
\end{equation}
The condition for the existence of the integral is then 
\begin{equation}\label{nc}
n< 2d/\mu,
\end{equation}
so that the exponent in Eq. \eqref{regioncontrib} is positive.
We also note that the moments scale quadratically. Let $t$ denote the radius of $\mathcal{D}(t).$ Then, 
\begin{equation}
{\bf E}\Bigl[\Bigl(\int_{\mathcal{D}(t)} M_\mu(dx)\Bigr)^n\Bigr] \thicksim const\, t^{nd-\mu n(n-1)/2}, \; t\rightarrow 0.
\end{equation}
This means that the multifractal spectrum of the measure is
\begin{equation}\label{mspec}
\zeta(q)=qd-\frac{1}{2}\mu q (q-1).
\end{equation}
Relying on the theory of Bacry and Muzy \cite{BM} and Robert and Vargas \cite{RV}, the measure is non-degenerate provided
\begin{equation}
\zeta'(1)>0,
\end{equation}
and the positive moments for $q>1$ are finite if
\begin{equation}
\zeta(q)>d.
\end{equation}
The first condition gives us 
\begin{equation}\label{bc}
\mu_c = 2d,
\end{equation}
and the second recovers Eq. \eqref{nc}.
These conditions are well-known in the case of the Bacry-Muzy GMC measure on the interval and circle.
Throughout this paper it is tacitly assumed that $\mu<2d,$ \emph{i.e.} we are in the sub-critical regime.

In the special case of the GMC measure in dimension $d=1$ we will consider a slightly different construction that is appropriate
for defining the GMC measure on the circle. 
Let 
\begin{align}\label{covr}
{\bf{Cov}}\left[\omega_{\mu,\varepsilon}(s), \,\omega_{\mu, \varepsilon}
(t)\right] &  =
\begin{cases}
- \, \mu\log r(s-t), \, \varepsilon\leq|s-t|< 1, \\
\mu\left(-\log r(\varepsilon) + \bigl(1-\frac{|s-t|}{\varepsilon}\bigr) \varepsilon\frac{d}{d\varepsilon} \log r(\varepsilon)\right),\, |s-t|\leq \varepsilon.    
\end{cases}
\end{align}
The process $\omega_{\mu, \varepsilon}(t)$ is defined on $t\in (0,1).$ Its mean is defined as in Eq. \eqref{mean}. 
The function $r(t)$ is assumed to
have the following properties.\footnote{Positive definiteness in the interval case follows from the other conditions provided
$d^2/ dt^2 \log r(t)<0$ and $d /dt|_{t=1} \log r(t)>0,$ 
cf. Appendix A. Also, in the circular case, the condition $\varepsilon\leq|s-t| < 1$ is replaced with
$\varepsilon\leq|s-t|\leq 1-\varepsilon.$ }
\begin{gather}
r(t) \; \text{is smooth and even} \; \text{on} \;(-1,\,0) \cup (0,\,1),  \\
\lim\limits_{t\rightarrow 0}\;t\frac{d}{dt} \log r(t)=1, \label{derivbound}\\
(s, t)\rightarrow -\log r(s-t) \; \text{is positive definite}.
\end{gather}
We will also impose one of the two boundary conditions,
\begin{subequations}
\begin{align}
r(1)&=1, \;\text{or} \label{boundline}\\
r(t)&=r(1-t). \label{boundcircle}
\end{align}
\end{subequations}
corresponding to the process being defined on the interval or circle, respectively. 
The two main examples are 
\begin{align}
r(t) = & |t|, \label{int}\\
r(t) = &  |1-e^{2\pi it}|.\label{cir}
\end{align}
The first is the Bacry-Muzy process \cite{BM} and the second is its circular version first considered heuristically in \cite{FyoBou} and rigorously in \cite{Jones}. 
It is shown in Appendix A that the gaussian process defined in Eq. \eqref{covr} can be constructed by properly generalizing the construction of Bacry and Muzy. Finally, one defines the limit measure to be the exponential functional of $\omega_{\mu,\varepsilon}(s)$ as above, with
$g(t)=-\log r(t).$ The formula for the
moments then takes on the form
\begin{equation}\label{momentsint1d}
S_n[\varphi](\mu) =  \int\limits_{[0,\,1]^n} \prod_{i=1}^n \varphi(s_i)\,
\prod\limits_{i<j}^n r(s_i-s_j)^{-\mu} ds_1\cdots ds_n, \; n<2/\mu.
\end{equation}

\section{Results}
\noindent
In this section we will present a formal, \emph{i.e.} exact at the level of formal power series, theory of intermittency differentiation and renormalization for a general class of random measures that
are defined to be (the $\varepsilon\rightarrow 0$ limit of) the exponential functional of a (regularized) gaussian process. The derivations of the main results in this section will be given in Section 4.

Consider a centered gaussian process $X(y)$ on $\mathbb{R}^d$ with covariance $g(x,\,y),$ 
possibly depending on $\varepsilon,$ which we drop from the list of arguments for brevity.
Define 
\begin{equation}\label{omegaX}
\omega_\mu(y)\triangleq \sqrt{\mu} X(y) - \frac{\mu}{2} {\bf Var} X(y).
\end{equation}
The corresponding random measure is defined by (the $\varepsilon\rightarrow 0$ limit of)
\begin{equation}
M_\mu(dy)= e^{\omega_\mu(y)} \, dy.
\end{equation}
The moments of the total mass of some compact region $\mathcal{D}$ are then given formally by the multiple integrals
\begin{equation}\label{momentsintgen}
S_n[\varphi](\mu) =  \int\limits_{\mathcal{D}^n} \prod_{i=1}^n \varphi(x_i)\,
\exp\Bigl(\mu \sum_{i<j}^n g(x_i, \,x_j)\Bigr) dx_1\cdots dx_n, 
\end{equation}
cf. Eq. \eqref{momentsint} above, up to some critical moment, beyond which they become infinite.
For example, the GMC construction of the previous section corresponds to
\begin{equation}\label{gtheta}
g(x, \,y) =  \bigl(\theta_\varepsilon\star g\bigr)(x-y),
\end{equation}
where $g(x)$ is defined in Eq. \eqref{glog}.  The main problem that we wish to tackle is how to characterize the mass of the limit random measure
from its divergent moments, \emph{i.e.} that of renormalization. 

Consider the general functional of the total mass of the form
\begin{equation}\label{thefunctional}
v(\mu,\,f,\,F) \triangleq {\bf E}\Bigl[F\Bigl(\int_\mathcal{D} e^{\mu
f(x)}\varphi(x) \,M_\mu(dx)\Bigr)\Bigr],
\end{equation}
where $f(x)$ are $F(x)$ are sufficiently smooth but otherwise arbitrary.\footnote{$\varphi(s)$ is fixed and dropped from the list of arguments for brevity.}
It is understood that the integration with
respect to $M_\mu(dx)$ is in the sense of
$\varepsilon\rightarrow 0$ limit. We will be primarily interested in the special case of
\begin{equation}
f_n(x) = \sum_{j=1}^n g(x, x_j)
\end{equation}
given some distinct $x_j,$ $j=1\cdots n.$ Define the corresponding functional
\begin{equation}\label{generfunctional}
v(\mu, F, x_1\cdots x_n)\triangleq  {\bf E}\Bigl[F\Bigl(\int_\mathcal{D} \varphi(x) \, M_{\mu}(dx)\Bigr)  
e^{ \omega_{\mu}(x_1)+\cdots+\omega_{\mu}(x_n)}\Bigr].
\end{equation}
Due to the Girsanov theorem identity, cf. Lemma \ref{girsan} below, the two functionals are related by
\begin{align}
v(\mu, F, x_1\cdots x_n) = \exp\Bigl(\mu\sum\limits_{i<j}^n 
g(x_i,\,x_j)\Bigr) \, v(\mu,\,f_n,\,F). \label{GirsanovGaussianLim}
\end{align}
Our first result is the rule of intermittency differentiation extending the corresponding result for the Bacry-Muzy GMC measure on the interval, cf. \cite{Me2}--\cite{MeIMRN},
in the form of a functional Feynman-Kac equation, in which the intermittency plays the role of time.
\begin{theorem}[Rule of Intermittency Differentiation]\label{theoremdiff}
The expectation $v(\mu,\, f,\,F)$ is invariant under
intermittency differentiation and satisfies
\begin{align} \frac{\partial}{\partial \mu} v(\mu,\,
f,\,F) & = \int\limits_{\mathcal{D}} v\bigl(\mu,\,f+g(\cdot,
x),\,F^{(1)}\bigr) e^{\mu f(x)}f(x)\varphi(x)\,dx+ \nonumber \\
& +\frac{1}{2}\int\limits_{\mathcal{D}^2} v\bigl(\mu,
f+g(\cdot,x_1)+g(\cdot,x_2),F^{(2)}\bigr)
e^{\mu\bigl(f(x_1)+f(x_2)+g(x_1,x_2)\bigr)} \times \nonumber \\ & \times g(x_1,
x_2)\varphi(x_1)\varphi(x_2)\,dx_1\,dx_2.\label{therule}
\end{align}
The expectation $v(\mu, F, x_1\cdots x_n)$ is also invariant under
intermittency differentiation and satisfies
\begin{align}
\frac{\partial}{\partial \mu} v(\mu,\,F, \,x_1\cdots x_n)  & =  v(\mu,\,F, \,x_1\cdots x_n) \, \sum\limits_{i<j}^n g(x_i,\ x_j)  + \nonumber \\ & + 
\int\limits_{\mathcal{D}} v\bigl(\mu,\,F^{(1)}, \,x_1\cdots x_{n+1}\bigr) \sum_{j=1}^n g(x_j, x_{n+1})\varphi(x_{n+1})\,dx_{n+1}+ \nonumber \\
 & + 
\frac{1}{2}\int\limits_{\mathcal{D}^2}  v\bigl(\mu,\,F^{(2)}, \,x_1\cdots x_{n+2}\bigr) \,g(x_{n+1}, x_{n+2})\times \nonumber \\ & \times \varphi(x_{n+1}) \varphi(x_{n+2})\, dx_{n+1}\,dx_{n+2}. \label{therulen}
\end{align}
\end{theorem}
The mathematical content of Theorem \ref{theoremdiff} is that differentiation with
respect to the intermittency parameter $\mu$ is equivalent to a
combination of two functional shifts induced by the $g$ function. It
is clear that the terms on the right-hand side of both Eqs. \eqref{therule} and \eqref{therulen}
are of the same functional form
as the corresponding functional on the left-hand side so that Theorem \ref{theoremdiff} allows us
to compute derivatives of all orders. We note that Eq. \eqref{therulen} is a special case
of Eq. \eqref{therule}, cf. Eq. \eqref{GirsanovGaussianLim} above. 
\begin{remark}
The interest in Eq. \eqref{therulen}
is that the functional $v(\mu,\,F, \,x_1\cdots x_n)$ is the ``minimal'' functional which is invariant under differentiation
and that Eq. \eqref{therulen}, unlike Eq. \eqref{therule}, is local. We conjecture that a hierarchy $\{v(\mu,\,F, \,x_1\cdots x_n)\}_n$ 
of functions that satisfies Eq. \eqref{therulen} captures the distribution uniquely. We also note that the three-term recurrences in Eqs. \eqref{therulen} and
\eqref{three_term} point to a possible connection with continuous fraction theory, which plays an important role in the classical moment problem, suggesting
that the non-classical moment problem of GMC might be analyzed by means of these recurrences. 
\end{remark}
As an immediate corollary of 
Theorem \ref{theoremdiff} we obtain an explicit formula for the $n$th intermittency derivative by a simple induction argument.
\begin{corollary}
\begin{equation}
\frac{\partial^n}{\partial \mu^n} v(\mu,\,f,\,F) =  \sum_{k=1}^{2n}
\int\limits_{\mathcal{D}^k} v\bigl(\mu,\,f+\sum\limits_{i=1}^k
g(\cdot,x_i), \,F^{(k)}\bigr) e^{\mu\bigl(\sum\limits_{i<j}^k
g(x_i,x_j)+\sum\limits_{i=1}^k
f(x_i)\bigr)} h_{n,k}(x)\,dx 
\end{equation}
for some functions $h_{n,k}(x)\equiv h_{n,k}(x_1,\cdots,x_k),$
$k=1\cdots 2n,$ that
are computed iteratively {\it via} the following three-term
recurrence\footnote{Empty sums and $h_{n,k}(x)$ for the values
of $k$ outside of $k=1\cdots 2n$ are understood to mean zero.}
\begin{gather}
h_{n+1,k}(x_1\cdots x_k)=\frac{1}{2} h_{n,k-2}(x_1\cdots x_{k-2})g(x_{k-1},x_k)\varphi(x_{k-1})\varphi(x_{k})+\nonumber \\ +
h_{n,k-1}(x_1\cdots x_{k-1}) \Bigl(\sum_{i=1}^{k-1} g(x_i,x_k)+f(x_k)\Bigr) \varphi(x_{k})+\nonumber \\ +
h_{n,k}(x_1\cdots x_k) \Bigl(\sum_{i<j}^k
g(x_i,x_j)+\sum_{i=1}^k f(x_i)\Bigr) , \label{three_term}
\end{gather}
for $k=1\cdots 2n+2,$ starting with
\begin{equation}
h_{1,1}(x)=f(x)\varphi(x),\,\, h_{1,2}(x_1,x_2)=\frac{1}{2}g(x_1,x_2)\varphi(x_1)\varphi(x_2).
\end{equation}
If $f\equiv 0,$ then the range of $k$ is changed to $2\cdots 2n.$ In particular, 
\begin{equation}
\frac{\partial^n}{\partial \mu^n}\Big\vert_{\mu=0}  {\bf E}\Bigl[F\Bigl(\int_\mathcal{D} \varphi(x) \,M_\mu(dx)\Bigr)\Bigr] =
\sum_{k=2}^{2n} \,
\int\limits_{\mathcal{D}^k} F^{(k)}\Bigl(\int_\mathcal{D} \varphi(z) dz\Bigr) h_{n,k}(x)\,dx. 
\end{equation}
\end{corollary}

The structure of the intermittency differentiation rule implies that one can represent the solution in the form of an expansion in $\mu,$ at least
formally. As an immediate corollary of Theorem \ref{theoremdiff} we obtain the following expansion (cf. \cite{Me3}  for the special case of the Bacry-Muzy GMC measure on the interval). Let
\begin{align}
\bar{\varphi} & \triangleq \int\limits_\mathcal{D} \varphi(x) \,dx, \\
H_{n,k}(\varphi) & \triangleq \int_{\mathcal{D}^k} h_{n,k}(x)\,dx,
\end{align}
and recall the formula for positive integer moments in Eq. \eqref{momentsintgen}.
\begin{theorem}[Intermittency Expansion]\label{ExpansionTh}
The total mass  ${\bf E}\Bigl[F\bigl(\int_\mathcal{D}\varphi(x)
\,M_\mu(dx)\bigr)\Bigr]$ has the formal intermittency expansion
\begin{equation}\label{theexpansion}
{\bf E}\Bigl[F\Bigl(\int_\mathcal{D} \varphi(x)
\,M_\mu(dx)\Bigr)\Bigr]=F(\bar{\varphi})+ \sum\limits_{n=1}^\infty
\frac{\mu^n}{n!}\Bigl[ \sum_{k=2}^{2n} F^{(k)}(\bar{\varphi})
H_{n,k}(\varphi)\Bigr].
\end{equation}
The expansion coefficients $H_{n, k}(\varphi)$ 
are given by the binomial transform of the derivatives of the
positive integer moments at zero intermittency,
\begin{equation}\label{HnkS}
H_{n, k}(\varphi) = \frac{(-1)^k}{k!} \sum\limits_{l=2}^k (-1)^l
\binom{k}{l}\,\bar{\varphi}^{k-l}\,\frac{\partial^n S_l}{\partial
\mu^n}[\varphi]\Big\vert_{\mu=0},
\end{equation}
\end{theorem}
It is clear that the expansion coefficients $H_{n, k}(\varphi)$ are determined uniquely by the moments. The reason for this is that 
they are independent of $F(x),$ so one can in particular take $F(x)=x^n$ and then use binomial inversion to compute them.
\begin{remark}
The intermittency expansion is the high temperature expansion of the distribution of the total mass
and is naturally interpreted as 
the asymptotic expansion in intermittency in the limit $\mu\rightarrow 0.$ Its structure can be elucidates further
by means of the following remarkable property that follows from the structure of the multiple integral representation of
the moments in Eq. \eqref{momentsintgen}. 
\end{remark}
\begin{theorem}[Intermittency Renormalization]\label{IntRenormal}
The expansion coefficients $H_{n, k}(\varphi)$ as defined by Eq. \eqref{HnkS} satisfy the fundamental renormalizability identity 
\begin{equation}\label{Hkeyid}
H_{n, k}(\varphi) = 0\,\,\,\forall k>2n.
\end{equation}
\end{theorem}
\begin{corollary}
The intermittency expansion in Eq. \eqref{theexpansion}
is an exactly renormalized expansion in the centered moments of
$\int_\mathcal{D} \varphi(x) \,M_\mu(dx).$ 
\begin{equation}
{\bf E}\Bigl[F\bigl(\int_\mathcal{D} \varphi(x)
\,M_\mu(dx)\bigl)\Bigr] =  F(\bar{\varphi})+ \sum\limits_{n=1}^\infty
\frac{\mu^n}{n!}  \Bigl[ \sum_{k=0}^{\infty} \frac{F^{(k)}(\bar{\varphi})}{k!} 
\frac{\partial^n}{\partial \mu^n}\Big\vert_{\mu=0} {\bf
E}\bigl[(\int_\mathcal{D} \varphi(x) \,M_\mu(dx)-\bar{\varphi})^k\bigr] \Bigr]. \label{renormexpans}
\end{equation}
\end{corollary}
This result follows from a simple observation that
\begin{equation}
\frac{\partial^n}{\partial \mu^n}\Big\vert_{\mu=0} \,{\bf
E}\Bigl[\Bigl(\int_\mathcal{D} \varphi(x) \,M_\mu(dx)-\bar{\varphi}\Bigl)^k\Bigr]=
k!\,H_{n,k}(\varphi)
\end{equation}
As the $k$ sum in Eq. \eqref{renormexpans} is finite by Eq. \eqref{Hkeyid}, we see that Eq. \eqref{renormexpans} is equivalent to Eq. \eqref{theexpansion}.
Of course, if the moments were all finite, then one could also write
\begin{equation}
{\bf E}\Bigl[F\bigl(\int_0^1 \varphi(s)
\,M_\mu(ds)\bigl)\Bigr] =  F(\bar{\varphi})+ \sum\limits_{n=1}^\infty
\frac{\mu^n}{n!} \frac{\partial^n}{\partial \mu^n}\Big\vert_{\mu=0}  \Bigl[ \sum_{k=0}^{\infty} \frac{F^{(k)}(\bar{\varphi})}{k!}
{\bf
E}\bigl[(\int_\mathcal{D} \varphi(x) \,M_\mu(dx)-\bar{\varphi})^k\bigr] \Bigr], \label{fullexpans}
\end{equation}
which is the naive expansion in the centered moments. This is not possible in our case. The formal equivalence of the expansion in
Eq. \eqref{renormexpans} to that in Eq. \eqref{fullexpans} shows that we have removed infinity from the moments and so found
an exactly renormalized solution. 
\begin{remark}
The fundamental open problem is to determine convergence properties of the intermittency expansion in Theorem \ref{ExpansionTh}.
We conjecture that it is convergent for smooth test functions of the GMC measure. This question is particularly
important in the context of identifying the law of GMC uniquely: does a probability distribution having the same moments
and coefficients $H_{n, k}(\varphi)$ (derivatives of the centered moments at zero intermittency) 
equal that of the total mass? 
\end{remark}
The full intermittency expansion can be calculated in closed form in terms of log-moments of the total mass. 
Let us assume that the expansion of the logarithm of the moments
in $\mu$ is known. In other words,
\begin{equation}\label{cpl}
\log \int\limits_{\mathcal{D}^l} \prod_{i=1}^l \varphi(x_i)\,
\exp\Bigl(\mu \sum_{i<j}^l g(x_i-x_j)\Bigr) dx_1\cdots dx_l = l\log \bar{\varphi}+ \sum\limits_{p=1}^\infty c_p(l) \mu^p,
\end{equation}
where it is understood that the coefficients $c_p(l)$ depend on $\varphi.$ We will first consider the special case of the Mellin transform. 
The intermittency expansion in this case takes on the form
\begin{gather}
{\bf E}\Bigl[\Bigl(\int_\mathcal{D} \varphi(x)
\,M_\mu(dx)\Bigr)^q\Bigr]=\bar{\varphi}^q+ \sum\limits_{n=1}^\infty
\frac{\mu^n}{n!}
f_n(q), \\
f_n(q) = \sum\limits_{k=2}^\infty (q)_k
\bar{\varphi}^{q-k}\,H_{n, k}(\varphi), \,\,n=1,2,3,\cdots,\label{fnq}
\end{gather}
where the sum has been extended to infinity by Eq. \eqref{Hkeyid},
$(q)_k \triangleq q(q-1)(q-2)\cdots (q-k+1),$ and it is understood that 
the coefficients $f_n(q)$ depend in $\varphi.$ 
The following result generalizes the calculation of the high temperature expansion of the Mellin transform in \cite{Me4}.
\begin{theorem}[Intermittency Expansion of the Mellin Transform]\label{mellin}
The expansion coefficients of the Mellin transform satisfy
the recurrence\footnote{In \cite{Me4} and \cite{MeIMRN} this recurrence is written in terms of $b_r(q)=(r+1)\,c_{r+1}(q).$
}
\begin{equation}\label{fnplusoneidentity}
f_{n+1}(q) = n!\sum\limits_{r=0}^n
\frac{f_{n-r}(q)}{(n-r)!} \,(r+1)\,c_{r+1}(q),\; f_0(q)=\bar{\varphi}^q, 
\end{equation}
and the intermittency expansion of the Mellin transform is
\begin{equation}\label{intermittencyMellin}
{\bf E}\Bigl[\Bigl(\int_\mathcal{D} \varphi(x)
\,M_\mu(dx)\Bigr)^q\Bigr]=\bar{\varphi}^q \exp\Bigl(\sum_{r=1}^\infty
\mu^{r} \,c_r(q)\Bigr).
\end{equation}
\end{theorem}
Thus, the expansion for the complex moments is obtained by replacing $l$ with $q$ in the expansion of the positive integer moments.
\begin{remark}
It is a major open question to quantify the extent to which the intermittency expansion of the Mellin transform
determines the probability distribution. The expansion of the Mellin transform in $\mu$ is not expected to be convergent in general, except for finite ranges of integer moments, cf. Propositions 4.1 and 4.2 in \cite{Me4} for the Bacry-Muzy measure on the interval, and should be interpreted as the asymptotic expansion. We conjecture the following uniqueness result: a probability distribution having the same moments
and asymptotic expansion of the Mellin transform in intermittency as the moments of the mass of the GMC measure and its intermittency
expansion, respectively, coincides with the mass distribution. Such probability distributions are known for the Bacry-Muzy measures on the
interval \cite{Me4} and circle, cf. 
Section 5 below. What is lacking is an appropriate uniqueness result.
\end{remark}
The significance of Theorem \ref{mellin} is that it allows one to reconstruct the high temperature expansion of the Mellin transform 
from the dependence of the moments on $\mu.$
Moreover, one can compute the expansion of 
a general transform of the total mass, extending the corresponding result
for the Bacry-Muzy GMC measure on the interval, cf. \cite{Me5}.
\begin{corollary}[Intermittency Expansion of the General Transform]
Consider the normalized random variable
\begin{equation}
\widetilde{M}\triangleq
\frac{1}{\bar{\varphi}}
\int_\mathcal{D} \varphi(x) \,M_\mu(dx).
\end{equation}
Given constants $a$ and $s$ and a smooth function
$F(s),$ the intermittency expansion of the general transform of
$\log\widetilde{M}$
\begin{equation}
{\bf
E}\Bigl[F\bigl(s+a\log\widetilde{M}\bigr)\Bigr]
= \sum\limits_{n=0}^\infty
F_n\bigl(a,\,s\bigr) \frac{\mu^n}{n!}
\end{equation}
is determined by
$F_0\bigl(a,\,s)=F(s),$ and
\begin{equation}
F_{n+1}\bigl(a,\,s\bigr) =
\sum\limits_{r=0}^n \frac{n!}{(n-r)!}  \, (r+1) c_{r+1}\Bigl(a\frac{d}{ds}\Bigr)
F_{n-r}\bigl(a,\,s\bigr).
\end{equation}
\end{corollary}
This result shows that the solution for the general transform is obtained by replacing $q$ with $a\,d/ds$ in the solution
for the Mellin transform in Theorem \ref{mellin}.
\begin{remark}
The only cases, in which the coefficients $c_p(l)$ are known analytically, are those of the Bacry-Muzy GMC measure 
on the interval and circle and $\varphi(s)$ such that the moments are given by the Selberg or Morris
integral, respectively. In both cases $c_p(l)$ are expressed as functions of $l$ in terms of Bernoulli polynomials, cf. \cite{Me4} and Section 5 below.
\end{remark}

We end this section with an extension of the intermittency differentiation rule and theory of intermittency renormalization to
the joint distribution of the measure, which we first considered in the special case of the Bacry-Muzy GMC on the interval  in \cite{Me6}. 
From now on, we let $\varphi(s)=1$ for simplicity. 
Let $\mathcal{D}_j,$ $j=1\cdots N,$ denote $N$ non-overlapping compact subsets
of $\mathcal{D}$ and $F_j(x),$ $j=1 \cdots N,$ denote $N$ smooth functions.
Consider the functional 
\begin{equation}
v\bigl(\mu,\,\vec{f},\,\vec{F},\,\vec{\mathcal{D}}\bigr) \triangleq {\bf
E}\Bigl[\prod\limits_{j=1}^N F_j\Bigl(\int_{\mathcal{D}_j} e^{\mu f_j(x)}
M_\mu(dx)\Bigr)\Bigr].
\end{equation}
We will use $|\mathcal{D}|$ to denote the Lebesgue measure of $\mathcal{D},$ 
and write
$\vec{f}+\vec{g}(\cdot, s)$ to denote the vector function with
components $f_j(u)+g(u, s),$ $j=1\cdots N.$ 
Denote the joint positive integer moments of 
$\int_{\mathcal{D}_1} M_\mu(dx),$ $\cdots,$ $\int_{\mathcal{D}_N} M_\mu(dx)$
by
\begin{equation}
S_{q_1\cdots q_N}(\mu, \vec{I}) \triangleq {\bf
E}\Bigl[\prod\limits_{j=1}^N \Bigl(\int_{\mathcal{D}_j}
M_\mu(dx)\Bigr)^{q_j}\Bigr].
\end{equation}

\begin{theorem}[Intermittency Renormalization for Multiple Subsets]
The rule of intermittency differentiation for multiple subsets is
\begin{gather}
\frac{\partial}{\partial \mu}
v\bigl(\mu,\,\vec{f},\,\vec{F},\,\vec{\mathcal{D}}\bigr) = \sum\limits_{j=1}^N
\int\limits_{\mathcal{D}_j} f_j(x) e^{\mu f_j(x)}
v\Bigl(\mu,\vec{f}+\vec{g}(\cdot,
x),F_1\cdots F^{(1)}_j\cdots F_N, \vec{\mathcal{D}}\Bigr) dx+ \nonumber \\
+\frac{1}{2} \sum\limits_{j=1}^N \iint\limits_{\mathcal{D}^2_j}
\exp\Bigl(\mu\bigl(f_j(x_1)+f_j(x_2)+g(x_1,x_2)\bigr)\Bigr)\,
g(x_1,\,x_2) \star \nonumber \\
\star v\Bigl(\mu, \vec{f}+\vec{g}(\cdot,x_1)+\vec{g}(\cdot,x_2),
F_1\cdots F^{(2)}_j \cdots F_N,\vec{\mathcal{D}}\Bigr) \,dx_1dx_2+ \nonumber \\
 + \sum\limits_{j<k}^N \;\;\iint\limits_{\mathcal{D}_j\times \mathcal{D}_k}
\exp\Bigr(\mu\bigl(f_j(x_1)+f_k(x_2)+g(x_1,x_2)\Bigr) \,g(x_1,\,x_2)
\star \nonumber\\
\star v\Bigl(\mu, \vec{f}+\vec{g}(\cdot,x_1)+\vec{g}(\cdot,x_2),
F_1\cdots F^{(1)}_j \cdots F^{(1)}_k \cdots F_N,\vec{\mathcal{D}}\Bigr) \,dx_1
dx_2.
\end{gather}
The intermittency expansion of the vector $\int_{\mathcal{D}_1} M_\mu(dx),\cdots ,\int_{\mathcal{D}_N} M_\mu(dx)$
is
\begin{equation}
{\bf E}\Bigl[\prod\limits_{j=1}^N F_j\Bigl(\int\limits_{\mathcal{D}_j}
M_\mu(dx)\Bigr)\Bigr] = \prod\limits_{j=1}^N
F_j(|\mathcal{D}_j|)+\sum\limits_{n=1}^\infty \frac{\mu^n}{n!}\!\!\!\!
\sum\limits_{2\leq \sum\limits_{j=1}^N k_j \leq 2n} H_{n; k_1\cdots
k_N} \prod\limits_{j=1}^N F_j^{(k_j)}(|\mathcal{D}_j|).
\end{equation}
Given $2\leq k_1+\cdots+k_N\leq 2n,$ the expansion coefficients $H_{n; k_1\cdots k_N}$
satisfy 
\begin{align}
H_{n; k_1\cdots k_N} = & \frac{(-1)^{k_1+\cdots+k_N}}{k_1!\cdots k_N!}
\sum\limits_{\substack{q_j\leq k_j \\j=1\cdots N}}
(-1)^{q_1+\cdots+q_N} \prod\limits_{j=1}^N |\mathcal{D}_j|^{k_j-q_j}
\binom{k_j}{q_j} 
\frac{\partial^n}{\partial\mu^n}\Big\vert_{\mu=0} \, S_{q_1\cdots
q_N}(\mu, \vec{\mathcal{D}}), \nonumber \\
= & \frac{1}{ k_1!\cdots k_N!}\frac{\partial^n}{\partial\mu^n}\Big\vert_{\mu=0}\;{\bf
E}\Bigl[\prod\limits_{j=1}^N \Bigl(\int_{\mathcal{D}_j}
M_\mu(ds)-|\mathcal{D}_j|\Bigr)^{k_j}\Bigr]. \label{Hcoeffmultiple}
\end{align}
The expansion coefficients as defined in Eq. \eqref{Hcoeffmultiple} for all indices
$k_1,\cdots, k_N$ satisfy 
\begin{equation}
H_{n; k_1\cdots k_N} = 0 \,\,\text{if} \,\, \sum\limits_{j=1}^N
k_j>2n.
\end{equation}
\end{theorem}

Hence, as before, the intermittency expansion is an exactly renormalized expansion in the joint centered
moments. The joint moments can be represented in the form of multiple integrals similar to Eq. \eqref{momentsint}.
It is an open question how to compute these integrals, even for $g(x, y)=-\log|x-y|$ and $N=2,$ \emph{i.e.} the joint distribution
the Bacry-Muzy measure of two subinterval of the unit interval, 
cf. \cite{Me6} and \cite{MeLMP} for detailed discussions and partial analytical results.

\section{Derivations}
\noindent In this section we will give derivations of our main results. Most of the arguments are exact at the level of formal power series
but not mathematically rigorous as we generally shun questions of convergence. We begin with a proof of the differentiation rule in the form
of Eq. \eqref{therule}. A direct proof\footnote{
The interest in the direct proof of Eq. \eqref{therulen} is that it naturally generalizes to 
the infinitely divisible case, cf. \cite{Me17}, whereas the derivation of Eq. \eqref{therule} given in this section appears to be specific to the gaussian case.
} of Eq. \eqref{therulen} that does not involve the Girsanov theorem or Eq. \eqref{GirsanovGaussianLim} 
is given in Appendix B. The argument for Eq. \eqref{therule} 
is based on repeated applications of the Girsanov theorem in the following form. 
\begin{lemma}[Girsanov]\label{girsan}
Let $X$ be a centered gaussian process as in Section 3 and $g(x,\,y)$  denote its covariance. Given a general functional $G$ of $X,$ $\alpha>0,$  
and distinct $y_i,$ $i=1\cdots n,$ there holds the identities
\begin{align}
{\bf E}\Bigl[ G(X)\, \prod\limits_{i=1}^n e^{\alpha\,X(y_i)-\frac{\alpha^2}{2} {\bf Var} X(y_i)}\Bigr] = & e^{\alpha^2
\sum\limits_{i<j}^n g(y_i,\,y_j)}\,
{\bf E}\Bigl[ G\Bigl(X+\alpha\sum\limits_{i=1}^n g(\cdot,\, y_i)\Bigr)\Bigr], \label{gir1} \\
{\bf E}\Bigl[ G(X)\, X(y)\Bigr] = & \frac{\partial}{\partial \alpha}\Big\vert_{\alpha=0} {\bf E}\Bigl[ G\Bigl(X+\alpha g(\cdot,\, y)\Bigr)\Bigr].\label{gir2}
\end{align}
\end{lemma}
\begin{proof}
We have by construction
\begin{equation}
{\bf E}\bigl[e^{\alpha\,X(y)-\frac{\alpha^2}{2} {\bf Var} X(y)}\bigr]=1
\end{equation}
for all $y.$ 
Introduce an equivalent probability measure
\begin{equation}
d\mathcal{Q}\triangleq  e^{-\alpha^2
\sum\limits_{i<j}^n g(y_i,\,y_j)}\prod\limits_{i=1}^n e^{\alpha X(y_i)-\frac{\alpha^2}{2} {\bf Var} X(y_i)}
\,d\mathcal{P}
\end{equation}
where $\mathcal{P}$ is the original probability measure
corresponding to ${\bf E}.$ Then, the law of the process
$y\rightarrow X(y)+\alpha\sum_{i=1}^n g(y,\,y_i)$ with respect to $\mathcal{P}$
equals the law of the original process $s\rightarrow X(y)$ with
respect to $\mathcal{Q}.$ Indeed, it is easy to show that the two
processes have the same finite-dimensional distributions by
computing their characteristic functions. The computation is straightforward. 
The continuity of sample paths can then be used
to conclude that the equality of all finite-dimensional
distributions implies the equality in law. \qed
\end{proof}
We now proceed to the derivation of Theorem \ref{theoremdiff} in the form of Eq. \eqref{therule}.
\begin{proof}
Recall the definition of $v(\mu,\,f,\,F)$ in Eq. \eqref{thefunctional}. We have
\begin{align}
\frac{\partial}{\partial \mu} v(\mu,\,f,\,F) = & {\bf E}\Bigl[F'\Bigl(\int_\mathcal{D} e^{\mu
f(y)} \varphi(y) \,M_\mu(dy)\Bigr) e^{\sqrt{\mu} X(y) - \frac{\mu}{2} {\bf Var} X(y)+\mu f(y)} \times \nonumber \\ & \times
\bigl[ \frac{1}{2\sqrt{\mu}}X(y) -\frac{1}{2}{\bf Var} X(y) + f(y) \bigr] \Bigr] \varphi(y) dy.\label{first}
\end{align}
To simplify notation, we will write $\cdot$ for the arguments of $F'$ and $F''$ below. 
By Eq. \eqref{gir1} with $\alpha=\sqrt{\mu}$ and $n=1$ we have the identity
\begin{gather}
{\bf E}\Bigl[F'(\cdot) \,e^{\sqrt{\mu} X(y) - \frac{\mu}{2} {\bf Var} X(y)+\mu f(y)} 
\bigl[ -\frac{1}{2}{\bf Var} X(y) + f(y) \bigr] \Bigr] \varphi(y) dy = \nonumber \\
\int_\mathcal{D}  v(\mu,\,f+g(\cdot, y),\,F') \, e^{\mu f(y)} \bigl[ -\frac{1}{2}{\bf Var} X(y) + f(y) \bigr]  \varphi(y) dy.\label{second}
\end{gather}
The remaining term can be reduced by means of Eq. \eqref{gir2}. We have the identity
\begin{align}
{\bf E}\Bigl[F'(\cdot) \,e^{\sqrt{\mu} X(y)} X(y) \Bigr] = & \sqrt{\mu}\, \int_\mathcal{D} {\bf E}\Bigl[F''(\cdot) \, e^{\sqrt{\mu} X(z) - \frac{\mu}{2} {\bf Var} X(z)+\mu f(z)} e^{\sqrt{\mu} X(y)}\Bigr]  \times \nonumber \\ & \times g(z,\,y) \varphi(z) dz  + \sqrt{\mu}\,
{\bf E}\Bigl[F'(\cdot) \,e^{\sqrt{\mu} X(y)}\Bigr]  {\bf Var} X(y).\label{third}
\end{align}
It follows 
\begin{gather}
\frac{1}{2\sqrt{\mu}} \int_\mathcal{D} {\bf E}\Bigl[F'(\cdot)\, e^{\sqrt{\mu} X(y) - \frac{\mu}{2} {\bf Var} X(y)+\mu f(y)}
X(y) \Bigr] \varphi(y) dy = \nonumber \\ \frac{1}{2} \int_{\mathcal{D}^2} {\bf E}\Bigl[F''(\cdot) \, e^{\sqrt{\mu} X(z) - \frac{\mu}{2} {\bf Var} X(z)}
e^{\sqrt{\mu} X(y) - \frac{\mu}{2} {\bf Var} X(y)}\Bigr] e^{\mu f(z) + \mu f(y)} g(z,\,y)  \varphi(z) \varphi(y) \times \nonumber \\ 
\times dzdy + \frac{1}{2} \int_\mathcal{D}  v(\mu,\,f+g(\cdot, y),\,F') \, e^{\mu f(y)} {\bf Var} X(y)  \varphi(y) dy.\label{fourth}
\end{gather}
It remains to substitute this expression and Eq. \eqref{second} 
into Eq. \eqref{first}, notice the cancellation, and apply Eq. \eqref{gir1} with $n=2$ to the double integral expression in Eq. \eqref{fourth}. \qed
\end{proof}

We now proceed to Theorem \ref{IntRenormal}. 
\begin{proof}
The proof of Eq. \eqref{Hkeyid} follows from the structure of the formula for the moments. Recalling
Eqs. \eqref{momentsintgen} and \eqref{HnkS}, we can write
\begin{equation}
H_{n, k}(\varphi) = \frac{(-1)^k}{k!} \sum\limits_{l=2}^k (-1)^l \binom{k}{l} \bar{\varphi}^{k-l}
\int\limits_{\mathcal{D}^l}  \prod_{i=1}^l \varphi(x_i) \Bigl[\sum\limits_{i<j}^l g(x_i,
x_j)\Bigr]^n\,dx^{(l)}.
\end{equation}
The claim is that Eq. \eqref{Hkeyid} follows from this equation and the fact that $g(x, y)$ is symmetric. 
Let $\mathcal{S}^k_l$ denote the set of all subsets of
$\{1\cdots k\}$ consisting of exactly $l$ elements. Let
\begin{equation}
h_{n, k} (x) \triangleq \frac{(-1)^k}{k!}\prod_{i=1}^k \varphi(x_i)\sum\limits_{l=2}^k
(-1)^l \sum\limits_{\sigma\in\mathcal{S}^k_l}
\Bigl[\sum\limits_{\substack{i<j\\i,j\in\sigma}} g(x_i,
x_j)\Bigr]^n.
\end{equation}
Clearly, $h_{n, k}(x)$ is symmetric in $x_1\ldots x_k.$ The size of $\mathcal{S}^k_l$ is $\binom{k}{l}$ so that we have the identity for any $l\leq k,$
\begin{equation}
\int\limits_{\mathcal{D}^k}  \prod_{i=1}^k \varphi(x_i)  \sum\limits_{\sigma\in\mathcal{S}^k_l}
\Bigl[\sum\limits_{\substack{i<j\\i,j\in\sigma}} g(x_i,
x_j)\Bigr]^n \,dx^{(k)}=\binom{k}{l} \bar{\varphi}^{k-l} \!\!
\int\limits_{\mathcal{D}^l}  \prod_{i=1}^l \varphi(x_i) \Bigl[\sum\limits_{i<j}^l g(x_i,
x_j)\Bigr]^n \!dx^{(l)},
\end{equation}
and, therefore,
$H_{n,k}(\varphi) = \int\limits_{\mathcal{D}^k}
h_{n,k}(x)\,dx^{(k)}.$
The interest in the symmetrized representation is explained in the following proposition, cf. \cite{Me3}.
\begin{lemma}
Up to the prefactor $\prod_{i=1}^k \varphi(x_i) /k!,$ the coefficients $h_{n,k}(x)$
are the same as those terms in the multinomial expansion of
\begin{equation}\label{multinom}
\Bigl(\sum\limits_{i<j}^k g(x_i, x_j)\Bigr)^n
\end{equation}
that involve all the indices $x_1\cdots x_k.$
\end{lemma}
This is best illustrated by an example. Let $n=3$ and $k=4.$ Then,
abbreviating $g_{ij}\triangleq g(x_i,x_j),$ we have for $h_{3,4}$
\begin{gather}
3 g_{12}g_{34}^2+3g_{34}g_{12}^2+3 g_{13}g_{24}^2+3
g_{24}g_{13}^2+ 3g_{14}g_{23}^2+3g_{23}g_{14}^2 +6g_{12}g_{23}g_{34} + 6g_{12}g_{13}g_{14}\nonumber \\
+6g_{12}g_{23}g_{24}+6g_{13}g_{23}g_{34}+
6g_{14}g_{24}g_{34}+6g_{12}g_{13}g_{34}+6g_{12}g_{14}g_{23}+6g_{13}g_{14}g_{24}\nonumber\\
+6g_{12}g_{14}g_{34}+6g_{13}g_{14}g_{23}+6g_{12}g_{24}g_{34}+6g_{14}g_{23}g_{34}+6g_{14}g_{23}g_{24}\nonumber\\
+6g_{13}g_{24}g_{34} +6g_{13}g_{23}g_{24}+6g_{12}g_{13}g_{24}.
\end{gather}
It follows that the coefficients $H_{n,k}(\varphi)$ must satisfy Eq. \eqref{Hkeyid}
as each term in the multinomial expansion in Eq. \eqref{multinom} contains the
product of exactly $n$ factors, and each individual factor involves
two distinct indices. These $n$ factors together must involve all
the $k$ indices, which is only possible if $k\leq 2n.$ \qed

\end{proof}

We finally proceed to Theorem \ref{mellin}. 
The idea of the proof is to compute the expansion coefficients $H_{n,k}(\varphi)$ recursively and effectively sum the intermittency series in 
the special case of $F(x)=x^q,$ $q\in\mathbb{C}.$ We first need to establish the following property 
 of the coefficients $c_p(l)$ 
involved in the expansion of log-moments in intermittency, cf. Eq. \eqref{cpl}. 
\begin{lemma}\label{polyndepend}
The coefficients $c_p(l)$ are polynomials in the moment order $l.$ 
\end{lemma}
\begin{proof}
Recalling the generating function of Hermite polynomials, 
\begin{equation}
e^{xt-t^2/2} = \sum\limits_{n=0}^\infty H_n(x) \frac{t^n}{n!},
\end{equation}
we can formally write with $t=\sqrt{\mu\,{\bf Var} X(y)}$ and $x=X(y)/\sqrt{{\bf Var} X(y)},$ 
\begin{align}
\int_\mathcal{D}  e^{\sqrt{\mu} X(y) - \frac{\mu}{2} {\bf Var} X(y)} \, \varphi(y) dy = &
\sum\limits_{n=0}^\infty \frac{\mu^{n/2}}{n!} \!\!\int_\mathcal{D} H_n\bigl(\frac{X(y)}{\sqrt{{\bf Var} X(y)}}\bigr) 
[{\bf Var} X(y)]^{n/2} \, \varphi(y) dy, \nonumber \\
& \equiv \bar{\varphi}\, \Bigl(1+\sum\limits_{n=1}^\infty \frac{\mu^{n/2}}{n!} m_n(\varphi)\Bigr)
.
\end{align}
It follows that the $l$th moments of the total mass is then given in terms of the potential partition polynomials $C_{n,l},$
cf. Section 11.5 of \cite{Charal},
\begin{equation}
S_l[\varphi](\mu) = \bar{\varphi}^l\, \sum\limits_{k=0}^\infty \frac{\mu^{k/2}}{k!} {\bf E}\Bigl[ C_{k, l}\bigl(m_1(\varphi),\cdots,
m_k(\varphi)\bigr)\Bigr],
\end{equation}
where the polynomials $C_{n,l},$ are defined by the formal power series identity
\begin{equation}
\Bigl(1+\sum\limits_{n=1}^\infty m_n t^n/n!\Bigr)^l = \sum\limits_{k=0}^\infty  \frac{t^k}{k!} \, C_{k, l}(m_1,\cdots,m_k).
\end{equation}
It remains to observe that the potential polynomials $C_{n,l}$ are also polynomials in $l,$ cf. Eq. (11.27) in \cite{Charal}.
It follows that the coefficients of the expansion of the $l$th moment $S_l[\varphi](\mu)$ in $\mu$ are polynomials in $l.$
Finally, the result follows from expanding the log-moment $\log S_l[\varphi](\mu)$ in powers of $\mu$ using
the logarithmic potential polynomials, cf. Section 11.4 of \cite{Charal}. \qed
\end{proof}
The key element of the derivation of Eq. \eqref{fnplusoneidentity} is the recurrence of the expansion coefficients. 
Recall the definition of the complete exponential Bell polynomials $Y_n(x_1\cdots x_n),$
\begin{equation}
\exp\Bigl(\sum\limits_{k=1}^\infty x_k \frac{t^k}{k!}\Bigr) = \sum\limits_{n=0}^\infty Y_n(x_1\cdots x_n) \frac{t^n}{n!}.
\end{equation}
This is understood as the equality of formal power series. 
Then, by the definition of Bell polynomials, we have
\begin{equation}
\frac{\partial^n S_l}{\partial
\mu^n}[\varphi]\Big\vert_{\mu=0} = \bar{\varphi}^l\,Y_n\bigl(1! c_1(l),\cdots, n! c_n(l)\bigr).
\end{equation}
where it is understood that the coefficients $c_p(l)$ depend in $\varphi.$
The recurrence relation of the Bell polynomials, 
\begin{equation}\label{Bellrecurrence}
Y_{n+1} = \sum\limits_{r=0}^n \binom{n}{r} Y_{n-r} \,x_{r+1}, \; Y_0=1,
\end{equation}
implies the following recurrence for the expansion coefficients, cf. \cite{Me4}.
\begin{lemma}\label{reclemma}
The expansion coefficients
are uniquely determined by 
\begin{gather}
H_{n+1, k}(\varphi) = \bar{\varphi}^k A_{n, k}+
\sum\limits_{r=0}^{n-1}
\binom{n}{r}\sum\limits_{t=2}^k \bar{\varphi}^{k-t} H_{n-r, \,t}(\varphi)\, B_{r, \,t, \,k}, \,\,\,n\geq 0,\, k\geq 2, \\
A_{n, k} \triangleq  (-1)^k\frac{(n+1)!}{k!} \sum\limits_{l=2}^k (-1)^l\binom{k}{l} c_{n+1}(l), \\
B_{r, \,t, \,k} \triangleq (-1)^k
t!\frac{(r+1)!}{k!}\sum\limits_{l=t}^k (-1)^l
\binom{k}{l}\binom{l}{t} c_{r+1}(l).
\end{gather}
\end{lemma}
We can now put all the pieces together. 
\begin{proof}
The property of polynomial dependence of $c_p(l)$ on $l$ means that we can write 
\begin{equation}\label{cplpol}
c_p(l) = R_p\bigl(\frac{d}{dz}\bigr)\Big|_{z=0}  e^{zl}
\end{equation}
for a polynomial $R_p(z)$ of some degree depending on $p.$  Given the recurrence in Lemma \ref{reclemma} and the definition of 
$f_n(q)$ in Eq. \eqref{fnq}, we can write 
\begin{equation}\label{fnplusone}
f_{n+1}(q) = \bar{\varphi}^q\Bigl[\sum_{k=2}^\infty (q)_k\, A_{n,k}\Bigr] + \sum\limits_{r=0}^{n-1} \binom{n}{r} \sum_{t=2}^\infty \bar{\varphi}^{q-t} 
H_{n-r, \,t} \Bigl[\sum_{k=t}^\infty (q)_k \,B_{r, \,t, \,k}\Bigr].
\end{equation}
We will now show that the representation of coefficients in Eq. \eqref{cplpol} implies the identities
\begin{align}
\sum_{k=2}^\infty (q)_k \,A_{n,k} = & (n+1)! \,c_{n+1}(q), \label{Anksum}\\
\sum_{k=t}^\infty (q)_k \,B_{r, t, k} = & (q)_t (r+1)! \,c_{r+1}(q). \label{Brtksum}
\end{align}
In fact, substituting the expression for $A_{n,k}$ in Lemma \ref{reclemma} and using the identity
\begin{equation}
\sum\limits_{l=2}^k (-1)^l \binom{k}{l} e^{zl} = (1-e^z)^k+ke^z-1,
\end{equation}
we get the expression
\begin{equation}
\sum_{k=2}^\infty (q)_k \,A_{n,k} = (n+1)! \,  R_{n+1}\bigl(\frac{d}{dz}\bigr)\Big|_{z=0} \sum_{k=2}^\infty \frac{(-1)^k}{k!} (q)_k 
\bigl[(1-e^z)^k + k e^z -1\bigr],
\end{equation}
and Eq. \eqref{Anksum} follows from  the identity
\begin{equation}
\sum\limits_{k=2}^\infty \frac{a^k}{k!} (q)_k = (1+a)^q-qa-1.
\end{equation}
Similarly, using the identity
\begin{equation}
\sum\limits_{l=t}^k (-1)^l \binom{k}{l}\binom{l}{t} e^{zl}=(-1)^t
\binom{k}{t} e^{zt} (1-e^z)^{k-t},
\end{equation}
we obtain the expression
\begin{equation}
\sum_{k=t}^\infty (q)_k \,B_{r, \,t, \,k} = (r+1)! \, R_{r+1}\bigl(\frac{d}{dz}\bigr)\Big|_{z=0} e^{zt} \sum\limits_{k=t}^\infty (q)_k \frac{(-1)^{k-t}}{(k-t)!} (1-e^z)^{k-t},
\end{equation}
and Eq. \eqref{Brtksum} follows from the identity
\begin{equation}
\sum\limits_{k=t}^\infty \frac{a^{k-t}}{(k-t)!} (q)_k = (q)_t (1+a)^{q-t}.
\end{equation}
Having established Eqs. \eqref{Anksum} and \eqref{Brtksum}, it remains to observe that
Eq. \eqref{fnplusoneidentity} is now equivalent to Eq. \eqref{fnplusone} and Eq. \eqref{intermittencyMellin}
follows from the recurrence relation of Bell polynomials in Eq. \eqref{Bellrecurrence}.\qed
\end{proof}

\section{Bacry-Muzy GMC Measure on the Circle}
\noindent
In this section we will apply the general theory to the special case of the Bacry-Muzy GMC  on the circle. 
For our purposes in this section it is convenient to introduce the quantity
\begin{equation}
\tau\triangleq \frac{2}{\mu}.
\end{equation}
Let $r(t)$ be as in Eq. \eqref{cir}. 
Recall the Morris integral, cf. Chapters 3 and 4 of \cite{ForresterBook},
\begin{gather}
\int\limits_{[-\frac{1}{2},\,\frac{1}{2}]^n} \prod\limits_{l=1}^n  e^{ 2\pi i \theta_l(\lambda_1-\lambda_2)/2}  |1+e^{2\pi i\theta_l}|^{\lambda_1+\lambda_2} \,
\prod\limits_{k<l}^n |e^{2\pi i \theta_k}-e^{2\pi i\theta_l}|^{-2/\tau} \,d\theta = \nonumber \\ =\prod\limits_{j=0}^{n-1} \frac{\Gamma(1+\lambda_1+\lambda_2- \frac{j}{\tau})\,\Gamma(1-\frac{(j+1)}{\tau})}{\Gamma(1+\lambda_1- \frac{j}{\tau})\,\Gamma(1+\lambda_2- \frac{j}{\tau})\,\Gamma(1-\frac{1}{\tau})} \label{morris2}.
\end{gather}
To bring it to the form of Eq. \eqref{momentsint}, we let
\begin{equation}
\varphi(s) = |1-e^{2\pi i s}|^{2\lambda}, 
\end{equation}
for some $\lambda\geq 0$ so that the moments are
\begin{align}
S_n[\varphi](\mu) = & \int\limits_{[0,\,1]^n} \prod\limits_{l=1}^n  |1-e^{2\pi i\theta_l}|^{2\lambda} \,
\prod\limits_{k<l}^n |e^{2\pi i \theta_k}-e^{2\pi i\theta_l}|^{-\mu} \,d\theta, \\
= & \prod\limits_{j=0}^{n-1} \frac{\Gamma(1+2\lambda- \frac{j}{\tau})\,\Gamma(1-\frac{(j+1)}{\tau})}{\Gamma(1+\lambda - \frac{j}{\tau})^2\,\Gamma(1-\frac{1}{\tau})}.
\label{momlambda}
\end{align}
\begin{lemma}[Log Moment Expansion]\label{logmomentscircle}
The expansion Eq. \eqref{cpl} of the log moments in $\mu$ near $\mu=0$ is 
\begin{gather}
\log S_l[\varphi](\mu) = l\Bigl(\log\Gamma(1+2 \lambda)-2\log\Gamma(1+\lambda)\Bigr) + \sum\limits_{p=1}^\infty c_p(l) \mu^p, \\
c_p(l) = \frac{1}{p2^p} \Bigl[\bigl(\zeta(p,\,1+2\lambda)-2\zeta(p, 1+\lambda)\bigr)\frac{B_{p+1}(l)-B_{p+1}}{p+1}+\nonumber \\ +
\zeta(p)\frac{B_{p+1}(l+1)-B_{p+1}}{p+1}-l\zeta(p)\Bigr].
\end{gather}
\end{lemma}
\begin{proof}
This is a simple corollary of Eq. \eqref{momlambda} and the following formulas involving the Hurwitz zeta function and Bernoulli polynomials,
\begin{align}
\log\Gamma(a+z) = & \log\Gamma(a)+\sum\limits_{p=1}^\infty \frac{(-z)^p}{p} \zeta(p,a), \\
\sum\limits_{j=x}^y j^p = &\frac{B_{p+1}(y+1)-B_{p+1}(x)}{p+1},
\end{align}
and the convention $\zeta(1, a)=-\psi(a),$ $\zeta(p, 1)=\zeta(p).$ \qed
\end{proof}
By Theorem \ref{mellin} we know the asymptotic expansion of the Mellin transform. 
We now wish to construct a positive probability distribution having the properties that its moments are given by Eq. \eqref{momlambda}
and the asymptotic expansion of its Mellin transform coincides with the series in Theorem \ref{mellin}.
\begin{theorem}[Morris Integral Probability Distribution]\label{MorrisIntegralProb}
The function 
\begin{align}
\mathfrak{M}(q\,|\tau,\,\lambda_1,\,\lambda_2)=&\frac{\tau^{\frac{q}{\tau}}}{\Gamma^q\bigl(1-\frac{1}{\tau}\bigr)}
\frac{\Gamma_2(\tau(\lambda_1+\lambda_2+1)+1-q\,|\,\tau)}{\Gamma_2(\tau(\lambda_1+\lambda_2+1)+1\,|\,\tau)}
\frac{\Gamma_2(-q+\tau\,|\,\tau)}{\Gamma_2(\tau\,|\,\tau)}\times \nonumber \\ & \times
\frac{\Gamma_2(\tau(1+\lambda_1)+1\,|\,\tau)}{\Gamma_2(\tau(1+\lambda_1)+1-q\,|\,\tau)}
\frac{\Gamma_2(\tau(1+\lambda_2)+1\,|\,\tau)}{\Gamma_2(\tau(1+\lambda_2)+1-q\,|\,\tau)} \label{thefunction}
\end{align}
reproduces the product in Eq. \eqref{morris2} when $q=n<\tau$
and is the Mellin transform of the distribution
\begin{align}
M_{(\tau, \lambda_1, \lambda_2)} = & \frac{\tau^{1/\tau}}{\Gamma(1-1/\tau)} \,\beta^{-1}_{22}(\tau, b_0=\tau,\,b_1=1+\tau \lambda_1, \,b_2=1+\tau \lambda_2) \times \nonumber \\ & \times
\beta_{1,0}^{-1}(\tau, b_0=\tau(\lambda_1+\lambda_2+1)+1), \label{thedecompcircle}
\end{align}
where $\beta^{-1}_{22}(a, b)$ is the inverse Barnes beta of type $(2,2)$ and $\beta^{-1}_{1,0}(a, b)$ is the independent inverse Barnes beta of type $(1,0).$ In particular, $\log M_{(\tau, \lambda_1, \lambda_2)}$ is infinitely divisible.
\end{theorem}
We refer the reader to \cite{Me14} for a review of the double gamma function $\Gamma_2(z\,|\,\tau)$ and to \cite{Me16} for a review of Barnes beta
distributions. 
The proof is given in \cite{Me16}. In the special case of $\lambda_1=\lambda_2=0$ this result first appeared in \cite{FyoBou}.
\begin{theorem}[Asymptotic Expansion]
The asymptotic expansion of $\log\mathfrak{M}(q\,|\tau,\,\lambda_1,\,\lambda_2)$ in $\tau$ in the limit $\tau\rightarrow \infty$ is
\begin{gather}
\log\mathfrak{M}(q\,|\tau,\,\lambda_1,\,\lambda_2)\thicksim q\Bigl(\log\Gamma(1+\lambda_1+\lambda_2)-\log\Gamma(1+\lambda_1)-\log\Gamma(1+\lambda_2)\Bigr) + 
\nonumber \\ + \sum\limits_{p=1}^\infty 
\frac{1}{p\tau^p} \Bigl[\bigl(\zeta(p,\,1+\lambda_1+\lambda_2)-\zeta(p, 1+\lambda_1)-\zeta(p, 1+\lambda_2)\bigr)\frac{B_{p+1}(q)-B_{p+1}}{p+1}+\nonumber \\ +
\zeta(p)\frac{B_{p+1}(q+1)-B_{p+1}}{p+1}-q\zeta(p)\Bigr].
\end{gather}
\end{theorem}
\begin{proof}
The first step is to express $\mathfrak{M}(q\,|\tau,\,\lambda_1,\,\lambda_2)$ in terms of the Alexeiewsky-Barnes $G-$function.
We have the identity
\begin{align}
\mathfrak{M}(q\,|\tau,\,\lambda_1,\,\lambda_2)=&\frac{1}{\Gamma^q\bigl(1-\frac{1}{\tau}\bigr)}
\frac{G(\tau(\lambda_1+\lambda_2+1)+1\,|\,\tau)}{G(\tau(\lambda_1+\lambda_2+1)+1-q\,|\,\tau)}
\frac{G(\tau\,|\,\tau)}{G(-q+\tau\,|\,\tau)}\times \nonumber \\ & \times
\frac{G(\tau(1+\lambda_1)+1-q\,|\,\tau)}{G(\tau(1+\lambda_1)+1\,|\,\tau)}
\frac{G(\tau(1+\lambda_2)+1-q\,|\,\tau)}{G(\tau(1+\lambda_2)+1\,|\,\tau)},\label{MG}
\end{align}
cf. \cite{MeIMRN} for the relationship between $\Gamma_2(z\,|\,\tau)$ and $G(z\,|\,\tau).$
\begin{lemma}
Let
\begin{equation}
I(q\,|\,a, \tau) \triangleq \int\limits_0^\infty
\frac{dx}{x}\frac{e^{-ax}}{e^{x\tau}-1}
\Bigl[\frac{e^{xq}-1}{e^{x}-1}
 -q-\frac{(q^2-q)}{2}x\Bigr].
\end{equation}
Then,  $I(q\,|\, a\tau, \tau)$ has the asymptotic expansion 
\begin{equation}
I(q\,|\,a\tau, \tau) \thicksim \sum\limits_{r=1}^\infty \frac{\zeta(r+1,
\,1+a)}{r+1}\Bigl(\frac{B_{r+2}(q)-B_{r+2}}{r+2}\Bigr)/\tau^{r+1}
\end{equation}
in the limit $\tau\rightarrow +\infty$ and 
\begin{equation}
I(q\,|\,a, \tau) =
\log\frac{G(1+a+\tau\,|\,\tau)}{G(1-q+a+\tau\,|\,\tau)}
-q\log\Bigl[\Gamma\bigl(1+\frac{a}{\tau}\bigr)\Bigr]+
\frac{(q^2-q)}{2\tau}\psi\bigl(1+\frac{a}{\tau}\bigr).
\end{equation}
\end{lemma}
The proof is given in \cite{MeIMRN}.
It remains to apply this lemma to each of the four ratios of the $G-$functions in Eq. \eqref{MG}
and collect the terms. \qed
\end{proof}

The Morris integral probability distribution thus has the required properties and so is naturally conjectured to be the
distribution of the total mass of the Bacry-Muzy GMC measure on the circle, cf. Conjecture \ref{GMCCIR} below.

\section{Conjectures and Open Questions}
\noindent In this section we will present a number of key conjectures and some open questions that are associated with our work.

\begin{conjecture}[Uniqueness]\label{C1}
Let $\bar{\varphi}=\int_{\mathcal{D}} \varphi(x)\,dx.$ The intermittency expansion, 
\begin{equation}
{\bf E}\Bigl[F\bigl(\int_{\mathcal{D}} \varphi(x)
\,M_\mu(dx)\bigr)\Bigr]=F(\bar{\varphi})+ \sum\limits_{n=1}^\infty
\frac{\mu^n}{n!}\Bigl[ \sum_{k=2}^{2n} F^{(k)}(\bar{\varphi})
H_{n,k}(\varphi)\Bigr],
\end{equation}
is convergent for sufficiently smooth $F(x)$ in a neighborhood of $\mu=0$ and its sum coincides with 
the left-hand side. In other words, the coefficients of the expansion $H_{n,k}(\varphi)$ capture the distribution uniquely.
\end{conjecture}
It is known that the intermittency expansion of the Mellin transform is convergent for finite ranges of positive and negative integer moments of the Bacry-Muzy
measure on the interval, cf. Propositions 4.1 and 4.2 in \cite{Me4}.
\begin{conjecture}[Infinite Divisibility]
The distribution of 
\begin{equation}
\log \int_{\mathcal{D}} \varphi(x)
\,M_\mu(dx)
\end{equation}
is infinitely divisible. 
\end{conjecture}
This is known for the Selberg and Morris integral probability distributions, cf. \cite{MeIMRN} and \cite{Me16} and was first discovered
in the special case of $\lambda_1=\lambda_2=0$ in \cite{Me4}.

\begin{conjecture}[GMC on the Circle]\label{GMCCIR}
Let $M_\mu(ds)$ denote the GMC measure on the circle described in Section 5
and $M_{(\tau, \lambda, \lambda)}$ denote the special case of the Morris integral
probability distribution as in Theorem \ref{MorrisIntegralProb} with $\tau=2/\mu.$ Then, 
\begin{equation}
\int_0^1 |1-e^{2\pi i s}|^{2\lambda} \,M_\mu(ds) = M_{(\tau, \lambda, \lambda)}.
\end{equation}
\end{conjecture}
The reason for the restriction $\lambda_1=\lambda_2$ is that $M_{(\tau,\lambda_1,\lambda_2)}$ is real-valued whereas  
the generalized total mass corresponding to the full Morris integral  
is not in general, unless
$\lambda_1=\lambda_2.$ This conjecture first appeared in \cite{FyoBou} for $\lambda_1=\lambda_2=0$ and in \cite{Me16} in general. 

\begin{conjecture}[GMC on the Interval]\label{GMCINT}
Let $M_\mu(ds)$ denote the GMC measure on the interval with $r(t)$ as in Eq. \eqref{int} and $\tau=2/\mu.$  
Then, the distribution of 
\begin{equation}
\int_0^1 s^{\lambda_1}(1-s)^{\lambda_2} \,M_\mu(ds) 
\end{equation}
has the Mellin transform
\begin{gather}
\Bigl(\frac{2\pi\,\tau^{\frac{1}{\tau}}}{\Gamma\bigl(1-1/\tau\bigr)}\Bigr)^q\;
\frac{\Gamma_2(1-q+\tau(1+\lambda_1)\,|\,\tau)}{\Gamma_2(1+\tau(1+\lambda_1)\,|\,\tau)}\times
\nonumber \\ \times
\frac{\Gamma_2(1-q+\tau(1+\lambda_2)\,|\,\tau)}{\Gamma_2(1+\tau(1+\lambda_2)\,|\,\tau)}
\frac{\Gamma_2(-q+\tau\,|\,\tau)}{\Gamma_2(\tau\,|\,\tau)}
\frac{\Gamma_2(2-q+\tau(2+\lambda_1+\lambda_2)\,|\,\tau)}{\Gamma_2(2-2q+\tau(2+\lambda_1+\lambda_2)\,|\,\tau)},\label{SelbIntProbDidMellin}
\end{gather}
\emph{i.e.} is the Selberg integral probability distribution.
\end{conjecture}
We refer the reader to \cite{MeIMRN} and \cite{Me14} for the original construction of the Selberg integral probability distribution and to \cite{Me16} for review.
This conjecture first appeared in \cite{Me4} in the special case of $\lambda_1=\lambda_2=0$ and in \cite{FLDR} and \cite{MeIMRN} in general.

\begin{conjecture}[Self-duality of the Mellin Transform]
Let $\tau=2/\mu.$ The Mellin transform,
\begin{equation}
\mathfrak{M}(q\,|\,\tau) \triangleq {\bf E}\Bigl[\bigl(\int_0^1 M_\mu(ds)\bigr)^q\Bigr],
\end{equation}
is self-dual (involution invariant) under the transformation
\begin{gather}
\tau\rightarrow \frac{1}{\tau},\; q\rightarrow \frac{q}{\tau},   \\ 
\mathfrak{M}\bigl(\frac{q}{\tau}\,\Big|\,\frac{1}{\tau}\bigr) (2\pi)^{-\frac{q}{\tau}}\,\Gamma^{\frac{q}{\tau}}(1-\tau) \Gamma(1-\frac{q}{\tau}) = 
\mathfrak{M}(q\,|\,\tau) (2\pi)^{-q} 
\Gamma^{q}(1-\frac{1}{\tau}) \Gamma(1-q).\label{involutionint}
\end{gather}
\end{conjecture}
This is known for both Morris\footnote{This holds for $\mathfrak{M}(q\,|\tau,\,\lambda_1,\,\lambda_2)$ in Eq. \eqref{thefunction}
provided it is multiplied by $(2\pi)^q$ and the $\lambda$s transform $\lambda_i=\tau\lambda_i,$ cf. \cite{Me16} for details.}
 and Selberg integral probability distributions and was first discovered in \cite{FLDR}, 
cf. also \cite{FLD}, \cite{Me14}, \cite{Me16} for extensions to nonzero $\lambda_1,\lambda_2$ and 
\cite{CRS} for self-duality in the model of the 2D Gaussian Free Field (GFF) restricted to circles, cf. Eq. \eqref{GFFoncircle} below. 


\begin{conjecture}[Addition of $O(\varepsilon)$ terms in Covariance]
If a gaussian process $\widetilde{\omega}_{\mu,\varepsilon}(s)$ has covariance that satisfies Eq. \eqref{covr} up to $O(\varepsilon)$ in the limit
$\varepsilon\rightarrow 0$ and has the mean that satisfies Eq. \eqref{mean}, then its exponential functional gives rise to the same GMC measure as 
in Section 2. 
\end{conjecture}
This is motivated by mesoscopic statistics of Riemann zeroes that we considered in \cite{Menon}. 

\begin{conjecture}[Centered GMC Measure]
Considered the centered version of the underlying gaussian field,
\begin{equation}\label{centeredomega}
\widetilde{\omega}_{\mu, \varepsilon}(u) \triangleq \omega_{\mu, \varepsilon}(u)-\omega_{\mu, \varepsilon}(0).
\end{equation}
Let $-1/\mu-1/2<\Re(q)<2/\mu,$ then 
\begin{equation}
\lim\limits_{\varepsilon\rightarrow 0} e^{\mu(\log r(\varepsilon)-1)\frac{q(q+1)}{2}} 
{\bf E} \Bigl[\Bigl(\int_0^1 \varphi(u)\,e^{\widetilde{\omega}_{\mu, \varepsilon}(u)} du\Bigr)^q\Bigr] =  {\bf E} \Bigl[ \Bigl(\int_0^1 r(u)^{\mu q}\,\varphi(u)\, M_\mu(du)\Bigr)^q\Bigr]. \label{gen1}
\end{equation}
\end{conjecture}
This is motivated by conjectured mod-gaussian limit theorems and based on Girsanov's theorem
for gaussian fields, cf. \cite{Menon} for details. For example, combining this conjecture with Eq. \eqref{SelbIntProbDidMellin}, one obtains for the Bacry-Muzy GMC on the interval
\begin{gather}
\lim\limits_{\varepsilon\rightarrow 0} e^{\mu(\log \varepsilon-1)\frac{q(q+1)}{2}} 
{\bf E} \Bigl[\Bigl(\int_0^1 e^{\widetilde{\omega}_{\mu, \varepsilon}(u)} du\Bigr)^q\Bigr]  = 
\frac{ (2\pi \tau^{\frac{1}{\tau}})^{q}}{\Gamma^{q}\bigl(1-1/\tau\bigr)}
\frac{\Gamma_2(1+q+\tau\,|\,\tau)}{\Gamma_2(1+2q+\tau\,|\,\tau)} \times \nonumber \\  \times
\frac{\Gamma_2(1-q+\tau\,|\,\tau)}{\Gamma_2(1+\tau\,|\,\tau)}
\frac{\Gamma_2(-q+\tau\,|\,\tau)}{\Gamma_2(\tau\,|\,\tau)}
\frac{\Gamma_2(2+q+2\tau\,|\,\tau)}{\Gamma_2(2+2\tau\,|\,\tau)}. \label{M1}
\end{gather}

We also want to mention a few open questions. 
\begin{enumerate}
\item Are there any examples of $r(t)$ different from Eq. \eqref{int} and Eq. \eqref{cir} for which the moments in Eq. \eqref{momentsint} and the full intermittency
expansion can be computed in closed form and the intermittency expansion can be re-summed to give the Mellin transform of a valid probability distribution? This is particularly interesting for the model of the 2D Gaussian Free Field (GFF) restricted to circles
that was recently considered in \cite{CRS} and corresponds to 
\begin{equation}\label{GFFoncircle}
r(t) = \Big|\frac{1-e^{2\pi i t}}{1-qe^{2\pi i t}}\Big|
\end{equation}
for some $q\in (-1, 1).$
\item Is the distribution of the total mass always expressible in terms of Barnes beta distributions (as is the case of both Selberg and Morris integral distributions)?
\item What is the dependence structure of GMC measures? We showed that the dependence structure of the GMC measure can be recovered, 
in the sense of intermitency expansions, from the joint integer moments. 
For example, for two subintervals $I_1$ and $I_2$  and $r(t)=|t|$
one needs to calculate for all $n$ and $m$ 
\begin{equation}
{\bf E}\Bigl[\Bigl( \int_{I_1} M_\mu(ds) \Bigr)^{n}\Bigl(  \int_{I_2} M_\mu(ds)\Bigr)^{m}\Bigr] =  \int\limits_{I_1^{n}\times I_2^{m}} 
\prod\limits_{k<l}^{n+m} |x_k-x_l|^{-\mu} dx, 
\end{equation}
The calculation of such integrals presents a particular challenge, cf. \cite{Me6} and \cite{MeLMP} for details.
\item How to extend the theory to non-stationary GMC measures? The simplest non-trivial example is to replace $\omega_{\mu,\varepsilon}(s)$ with
the centered process 
$\omega_{\mu,\varepsilon}(s)-\omega_{\mu,\varepsilon}(0)$ as in Eq. \eqref{centeredomega} above. This is particularly interesting in the context of extrema of 
a regularized version of the fractional Brownian motion with zero Hurst index considered in \cite{FLD} and \cite{FKS}  and mesoscopic statistics of Riemann zeroes considered in \cite{Menon}. 
\item How to extend the theory of intermittency expansions to complex functionals of the total mass? For example, the Morris integral probability
distribution is defined and is real-valued for $\lambda_1\neq \lambda_2$ in Theorem \ref{MorrisIntegralProb}. Yet, the corresponding functional of
the total mass is complex-valued, unless $\lambda_1=\lambda_2,$ in spite of the fact the integer moments of the total mass are real-valued. 
Therefore, one needs to consider more general moments than just the integer
moments of the total mass and develop the corresponding theory of intermittency renormalization.

\end{enumerate}

\section{Conclusions}
We have presented a theory of renormalization of multi-dimensional GMC measures. The theory is based on the rule of intermittency differentiation. 
The rule is an exact 
functional equation that prescribes how to differentiate a general class of functionals of the GMC measure
with respect to intermittency. A repeated application of this rule leads to a perturbative expansion of the functional in intermittency
known as the intermittency expansion (or the high temperature expansion). The intermittency expansion is a renormalized expansion in the centered moments
of the total mass of the GMC measure. We have shown that the full intermittency expansion of the Mellin transform can be computed for the whole
class of GMC measures considered in this paper provided one knows the expansion of positive integer moments, which effectively solves
the renormalization problem at the level of the high temperature expansion. 
We have illustrated the theory with the case of the periodized Bacry-Muzy GMC measure on the circle. 
We have explicitly computed the intermittency expansion, summed it in closed form, and showed that the resulting
Mellin transform is the Mellin transform of the Morris integral probability distribution, which is then conjectured to be
the distribution of the total mass on the circle.  

We have formulated two versions of the intermittency differentiation rule and given two separate formal (exact at the level of formal 
power series) derivations. The first version is non-local, \emph{i.e.} involves the measure of a continuum of subsets of the given 
set due to the very strong stochastic dependence of the measure, and its derivation is based 
on repeated application of the Girsanov theorem. The second version is in the form of an infinite hierarchy of local equations.
Its derivation is based on the intermittency invariance
of the underlying gaussian field, which we formulated for a general gaussian process in this paper. 
This invariance is a technical device
that substitutes for the non-existent Markov property of the underlying gaussian field and allows one to derive a Feynman-Kac equation
for the distribution of the total mass by considering a stochastic flow in intermittency (as opposed to time in the classical framework
of diffusions). The intermittency invariance gives two ways of evaluating the limit of the flow, which results in the differentiation rule.
The first way is the backward Kolmogorov equation, the second way involves detailed analysis of certain infinite series expansions,
combined with a combinatorial property of the measure. While the first derivation is simpler, the second does not rely on the use
of the Girsanov theorem and so has a natural generalization to the total mass of infinitely divisible multiplicative chaos measures. 

The interest in the intermittency differentiation rule goes beyond the calculational aspect of the ensuing intermittency expansions.
We have noted that the rule can be formulated in the form of a hierarchy of 
three-term recurrences and also gives rise to a three-term recurrence for the integrands involved in the multiple integral
representation of the expansion coefficients. These recurrences are of independent interest and deserve an in-depth study
as they suggest a connection with the continuous fraction theory and raise the possibility of analyzing the
non-classical moment problem of GMC by their means. Similarly, the theory of intermittency expansions is also of independent
theoretical interest for it provides a framework for posing the moment problem and 
leads to a number of conjectured solutions 
that we proposed in the paper. 
Our main conjecture is that the intermittency expansion captures the distribution of the total mass uniquely. 
Aside from the uniqueness problem, we have presented several other conjectures and formulated a number of open
questions that we hope will help stimulate future research on GMC measures. 




\section*{Acknowledgments}

The author wishes to express gratitude to the participants of the workshop ``Extrema of Logarithmically Correlated Processes, Characteristic Polynomials, and the Riemann Zeta Function'', Bristol, May 2016, for
drawing our interest to this project. 

\appendix
\section{A Review of the GMC on the Interval}
\noindent The construction of $\omega_{\mu, \varepsilon}(s)$ is based on the idea of using conical sets as in Bacry and Muzy \cite{BM} and modifying the intensity measure to match
the desired covariance. The conical sets live on the upper half-plane in the case of the boundary condition in Eq. \eqref{boundline} 
and on the torus in the case of Eq. \eqref{boundcircle}. 

The starting point is a gaussian independently scattered random measure
$P$ on the time-scale plane $\mathbb{H}_+=\{(t, \, l), \, \, l>0\},$
distributed uniformly with respect to some positive intensity measure $\rho.$
This means that $P(A)$ is a gaussian random variable for measurable subsets
$A\subset\mathbb{H}_+.$  The property of being independently
scattered means that $P(A)$ and $P(B)$ are independent if $A$ and
$B$ do not intersect. Uniform distribution with respect to $\rho$
means that the characteristic function of $P(A)$ is given by
\begin{equation}
{\bf{E}} \left[ e^{i q P(A)} \right]=e^{\mu\phi(q)\rho(A)},
\,\,\,q\in\mathbb{R},
\end{equation}
where $\mu>0$ is the intermittency parameter and $\phi(q)$ is the logarithm of the characteristic
function of the underlying gaussian distribution and is given by 
\begin{equation}
\phi(q) = -i\frac{q}{2}-\frac{q^2}{2}.
\end{equation}
The mean is chosen in such a way that
\begin{equation}
\phi(-i)=0 \,\,\, \text{so that}\,\,\,{\bf E}
\bigl[e^{P(A)}\bigr]=1\,\,\forall A\subset\mathbb{H}_+,
\end{equation}
which gives rise to the martingale property of the limit measure.
The existence of such random measures is established in \cite{RajRos}.
Next, following Bacry and Muzy \cite{BM}, we introduce special
conical sets $\mathcal{A}_{\varepsilon}(u)$ in the time-scale
plane defined by
\begin{equation}
\mathcal{A}_{\varepsilon}(u) = \left\{(t,l) \,\,\Big\vert\,\,
|t-u|\leq\frac{l}{2} \,\,\text{for}\,\,\varepsilon\leq l\leq 1
\,\,\text{and}\,\,|t-u|\leq\frac{1}{2}\,\,\text{for}\,\,l\geq 
1\right\}.
\end{equation}
The last
preparatory step is to define a family of gaussian processes with
dependent increments $\omega_{\mu, \varepsilon}(u)$ by
\begin{equation}\label{omegadef}
\omega_{\mu, \varepsilon}(u) =
P\left(\mathcal{A}_{\varepsilon}(u)\right).
\end{equation}
It is clear that $\omega_{\mu, \varepsilon}(u)$ and $\omega_{\mu,
\varepsilon}(v)$ are dependent in general if $|u-v|<1$ and are
independent otherwise so this case corresponds to $r(1)=1.$ 

We now wish to choose the intensity measure $\rho(dt\,dl)$ in such a way
that the process $\omega_{\mu, \varepsilon}(u)$ has the covariance given in Eq. \eqref{covr}.
We make the ansatz
\begin{equation}\label{ansatz}
\rho(dt \, dl)=\frac{f(l)}{l^2}\,dt\,dl.
\end{equation}

\begin{lemma}
Define $f(l)$ by 
\begin{gather}
\frac{f(l)}{l^2} = -\frac{d^2}{dl^2} \log r(l), \; l\in (0, 1). \label{fr} \\
f(l) = \frac{d}{dz}\Big\vert_{z=1} \log r(z),\;  l\geq 1.\label{fr2}
\end{gather}
and assume that these quantities are positive. Then, the process $\omega_{\mu, \varepsilon}(u)$ in 
Eq. \eqref{omegadef} with the intensity measure defined in Eq. \eqref{ansatz} has the covariance specified in Eq. \eqref{covr}.
\end{lemma}
Note that for $r(l)=l$ we recover the original result of Bacry and Muzy, namely,
\begin{equation}
f(l)=1.
\end{equation}
\begin{proof}
We need to compute the $\rho$ measure of the 
intersection of $\mathcal{A}_{\varepsilon}(u)$ and $\mathcal{A}_{\varepsilon}(v).$ 
Defining
\begin{equation}
g(z)=\int_z^1 \frac{f(l)}{l^2} dl,
\end{equation}
it is easy to show that we have the identity
\begin{equation}\label{gintegral}
\int_z^1 g(x)\,dx = \int_z^1 \frac{f(l)}{l^2}(l-z) dl.
\end{equation}
Using the definition of $f(l)$ in Eq. \eqref{fr} and Eqs. \eqref{fr2} and \eqref{boundline}, we have the additional identities
\begin{align}
g(z) & = -\frac{d}{dx}\Big\vert_{x=1} \log r(x) + \frac{d}{dz} \log r(z), \label{gcomputed}\\
\int_z^1 g(x)\,dx & = -(1-z)\frac{d}{dl}\Big\vert_{l=1} \log r(l) - \log r(z). \label{gintegralcomputed}
\end{align}
Assume first that $\varepsilon\leq|u-v|<1.$ Let $z=v-u,$ $u<v.$ Then, for $z\in[\varepsilon,\,1],$
\begin{equation}\label{totalarea}
\rho\Bigl(\mathcal{A}_{\varepsilon}(u)\cap \mathcal{A}_{\varepsilon}(v)\Bigr) = (1-z)\int_1^\infty \frac{f(l)}{l^2} dl +
\int_z^1 \frac{f(l)}{l^2}(l-z) dl.
\end{equation}
It follows from Eqs. \eqref{fr2}, \eqref{gintegral}, and \eqref{gintegralcomputed}, that the $\rho$ measure of the intersection equals $- \log r(z)$ as desired. 

Now, we need to consider the case of $z\leq \varepsilon.$ The $\rho$ measure of the intersection is
\begin{equation}
\rho\Bigl(\mathcal{A}_{\varepsilon}(u)\cap \mathcal{A}_{\varepsilon}(v)\Bigr) = (1-z)\int_1^\infty \frac{f(l)}{l^2} dl +
\int_\varepsilon^1 \frac{f(l)}{l^2}(l-z) dl.
\end{equation}
Using Eqs. \eqref{fr2} and \eqref{gintegral}, we can write for the $\rho$ measure of the intersection,
\begin{equation}
\rho\Bigl(\mathcal{A}_{\varepsilon}(u)\cap \mathcal{A}_{\varepsilon}(v)\Bigr) = \int_\varepsilon^1 g(x)\,dx
+(\varepsilon- z) g(\varepsilon) + (1-z) \frac{d}{dl}\Big\vert_{l=1} \log r(l).
\end{equation}
Finally, substituting Eqs. \eqref{gcomputed} and \eqref{gintegralcomputed} and noticing several cancellations, we get
\begin{equation}
\rho\Bigl(\mathcal{A}_{\varepsilon}(u)\cap \mathcal{A}_{\varepsilon}(v)\Bigr) = - \log r(\varepsilon) + \bigl(1-\frac{z}{\varepsilon}\bigr) \varepsilon\frac{d}{dl}\Big\vert_{l=\varepsilon} \log r(l).
\end{equation}
Hence, we arrive at the desired form of covariance in Eq. \eqref{covr}.  
\qed
\end{proof}

\section{Derivation of the Intermittency Differentiation Rule}
\noindent
In this section we will give a derivation of Eq. \eqref{therulen} that does not rely on the use of the Girsanov theorem. 
While admittedly more involved, the interest in this approach has to do with the fact that it naturally generalizes to
the infinitely divisible case, cf. \cite{Me17}, in which case the use of the Girsanov theorem appears to be significantly more
difficult. 

The derivation of Eq. \eqref{therulen} from first principles is essentially based on the \emph{intermittency invariance} of the underlying
gaussian field $\omega_{\mu, \varepsilon}(x),$ see \cite{Me2}, \cite{Me3}, and \cite{Me5}, \cite{Me17} for an extension to the general infinitely divisible field. 
Recall that the covariance of $\omega_{\mu, \varepsilon}(x)$ is of the form   $\mu\,g_\varepsilon(x, \,y),$ where
\begin{equation}
g_\varepsilon(x, \,y) =  \bigl(\theta_\varepsilon\star g\bigr)(x-y),
\end{equation}
cf. Eq. \eqref{gtheta}, for some singular positive definite function $g(x)$ such as in Eq. \eqref{glog}.
For our purposes in this section, it is convenient to define another gaussian field by
\begin{equation}
\omega_{\mu, L, \varepsilon}(x)\triangleq 
\omega_{\mu, \varepsilon}(x) + \mathcal{N}(-(\mu/2)\,\log L,\, \mu\log L),
\end{equation}
where $L\geq 1$ and $\mathcal{N}(-(\mu/2) \log L, \mu\log L)$ is an independent gaussian random variable with the mean
$-(\mu/2) \log L$ and variance $\mu\log L.$
It is obvious due to Eq. \eqref{bump} that the covariance of $\omega_{\mu, L, \varepsilon}(x)$ is of the same functional form as that
of $\omega_{\mu, \varepsilon}(x),$
\begin{align}
{\bf Cov}\bigl( \omega_{\mu, L, \varepsilon}(x), \, \omega_{\mu, L, \varepsilon}(y)\bigr) = & \mu\, \bigl(\theta_\varepsilon\star g_L\bigr)(x-y), 
\label{covgen} \\
g_L(x) = & g(x) + \log L.
\end{align}
\begin{lemma}[Intermittency Invariance]\label{keyinvlem}
Let $B(\delta)$ be the Brownian motion with drift $-\delta/2$ independent of
$\omega_{\mu,L,\varepsilon}(x).$ Then, we have the equality of gaussian processes in law,
\begin{equation}\label{interminvariance}
B(\delta)+\omega_{\mu,L,\varepsilon}(x)=\omega_{\mu-\delta,L,\varepsilon}(x)+
{\bar \omega}_{\delta,Le,\varepsilon}(x),
\end{equation}
which are viewed as random functions of $x$ 
at fixed
$0<\delta<\mu$ and $\varepsilon.$  ${\bar \omega}_{\delta,Le,\varepsilon}(x)$ denotes an
independent copy of the $\omega_{\mu,L,\varepsilon}(x)$ process at the
intermittency parameter $\delta$ and $L$ replaced with $Le,$ $e=2.718\cdots.$
\end{lemma}
\begin{proof}
It is sufficient to compare the means and covariances of the gaussian processes on the left- and right-hand sides of this equation.
The covariance of the left-hand side of Eq. \eqref{interminvariance} is 
\begin{align}
{\bf{Cov}}\Bigl(B(\delta)+\omega_{\mu, L, \varepsilon}(x), \,B(\delta)+\omega_{\mu, L, \varepsilon}
(y)\Bigr) & = \delta+\mu\,g_\varepsilon(x,\,y)+\mu\log L.
\end{align}
The covariance of the right-hand side of Eq. \eqref{interminvariance} is
\begin{gather}
{\bf{Cov}}\Bigl(\omega_{\mu-\delta, L, \varepsilon}(x) + {\bar \omega}_{\delta,Le,\varepsilon}(x), \,\omega_{\mu-\delta, L, \varepsilon}
(y)+{\bar \omega}_{\delta,Le,\varepsilon}(y)\Bigr) =  \nonumber \\ 
(\mu-\delta) \, \bigl(g_\varepsilon(x,\,y)+\log L\bigr) + \delta\, \bigl(g_\varepsilon(x,\,y)+\log (Le)\bigr),
\end{gather}
so the covariances are the same. We have for the mean
\begin{align}
{\bf{E}} \bigl[\omega_{\mu, L, \varepsilon}(x)\bigr]   = & -\frac{1}{2} \,
{\bf Var} \bigl[\omega_{\mu, L, \varepsilon}(x)\bigr], \nonumber \\
= &  -\frac{1}{2} \,
{\bf Var} \bigl[\omega_{\mu, \varepsilon}(x)\bigr] -\frac{\mu}{2}\log L 
\end{align}
so the means are also the same.
\qed
\end{proof}

We need two additional lemmas. 
It is convenient to introduce the quantity $g_{L,\varepsilon}(x,y)$ by
\begin{equation}
\mu\,g_{L,\varepsilon}(x,y) = {\bf{Cov}}\bigl(\omega_{\mu, L, \varepsilon}(x), \,\omega_{\mu, L, \varepsilon}
(y)\bigr).
\end{equation}
\begin{lemma}\label{lemmaexpec}
Let $\mathfrak{f}(\delta,x)$ be an arbitrary continuous function
that vanishes as $\delta\rightarrow 0.$ Let
$\mathcal{B}(x)\triangleq
e^{\mathfrak{f}(\delta,x)+\omega_{\delta,L,\varepsilon}(x)}-1.$
Then, given any distinct $x_1, \cdots, x_k,$ as
$\delta\rightarrow 0,$
\begin{gather}
{\bf E}[\mathcal{B}(x_1)\mathcal{B}(x_2)]=
\bigl(e^{\mathfrak{f}(\delta,x_1)}-1\bigr)\bigl(e^{\mathfrak{f}(\delta,x_2)}-1\bigr)+\delta\,g_{L,\varepsilon}(x_1,x_2)+o(\delta),\\
{\bf E} [\mathcal{B}(x_1)\cdots\mathcal{B}(x_k)] =
\bigl(e^{\mathfrak{f}(\delta,x_1)}-1\bigr)\cdots\bigl(e^{\mathfrak{f}(\delta,x_k)}-1\bigr)+o(\delta),\label{B2}
\,\, k\neq 2.
\end{gather}
\end{lemma}
\begin{proof}
Let $0\leq l\leq k,$ and $(p_1<\cdots <p_l)$ denote an increasing tuple of numbers from $1\cdots k$ of length $l$ 
and $\sum_{(p_1<\cdots < p_l)}$ denote the sum over all such $l-$tuples. It is easy to show that for any such
tuple 
\begin{equation}\label{auxexp}
{\bf E}\Bigl[
\exp\bigl(\sum\limits_{(p_1<\cdots <p_l)}\omega_{\delta, L, \varepsilon}(x_{p_i})\bigr)\Bigr]
= \exp\Bigl(\delta\,\sum\limits_{i<j}^l 
g_{L,\varepsilon}(x_{p_i},x_{p_j})\Bigr).
\end{equation}
Given $k$ distinct numbers, we have the algebraic identity
\begin{equation}\label{auxalg}
(a_1-1)\cdots (a_k-1) = \sum\limits_{l=0}^k (-1)^{k-l}
\sum\limits_{(p_1<\cdots <p_l)} \prod\limits_{i=1}^l a_{p_i},
\end{equation}
taking all empty sums to mean zero and empty products to mean one.
It is easily verified by induction. If we now expand the brackets on
the left-hand side of Eq. \eqref{B2} and make use of Eqs. \eqref{auxexp} and \eqref{auxalg},
we obtain
\begin{equation}
\sum\limits_{l=0}^k (-1)^{k-l} \sum\limits_{(p_1<\cdots <p_l)}
\exp\Bigl(\sum\limits_{i=1}^l \mathfrak{f}(\delta,x_{p_i})+\delta\sum_{i<j}^l g_{L,\varepsilon}(x_{p_i},x_{p_j})\Bigr).
\end{equation}
It remains to expand this expression in $\delta$ and recall that
$\mathfrak{f}(\delta,x)\rightarrow 0$ as $\delta\rightarrow 0$ by assumption. There
results
\begin{equation}\label{auxexp2}
\sum\limits_{l=0}^k (-1)^{k-l} \sum\limits_{(p_1<\cdots <p_l)}
e^{\sum\limits_{i=1}^l \mathfrak{f}(\delta,x_{p_i})} + \delta\sum\limits_{l=0}^k
(-1)^{k-l} \sum\limits_{(p_1<\cdots <p_l)}
\sum_{i<j}^l
g_{L,\varepsilon}(x_{p_i},x_{p_j})+o(\delta).
\end{equation}
By Eq. \eqref{auxalg}, the first term in Eq. \eqref{auxexp2} is exactly $\prod_{i=1}^k
\bigl(\exp(\mathfrak{f}(\delta,x_i))-1\bigr)$ that occurs on the right-hand side
of Eq. \eqref{B2}. It is not difficult to see that the second term in Eq. \eqref{auxexp2} equals $\delta \,g_{L,\varepsilon}(x_1,x_2)$ if $k=2$ and is zero
otherwise. \qed
\end{proof}
\begin{lemma}\label{lemmadiff}
Let 
$F(x)$ be smooth. 
Then, there holds the following identity
\begin{gather}
\frac{\partial}{\partial\mu} F\Bigl(z\int_\mathcal{D} e^{\omega_{\mu, L, \varepsilon}(x)}\,dx\Bigr)  \prod\limits_{i=1}^n e^{ \omega_{\mu,L,\varepsilon}(x_i) }  = 
-\lim\limits_{\delta\rightarrow 0} \frac{1}{\delta} \Bigl[\sum_{k=1}^\infty
F^{(k)}\Bigl(  z\int_\mathcal{D} e^{\omega_{\mu, L, \varepsilon}(x)}\,dx   \Bigr) \times \nonumber \\
\times
\frac{z^k}{k!}  \Bigl(\int\limits_\mathcal{D} e^{\omega_{\mu, L, \varepsilon}(x)}
\bigl(e^{\mathcal{A}_\varepsilon(x)}-1\bigr)\,dx\Bigr)^k  \sum\limits_{l=0}^n \sum\limits_{(p_1<\cdots <p_l)} \prod\limits_{i=1}^l \bigl(e^{\mathcal{A}_\varepsilon(x_{p_i})}-1\bigr) + \nonumber \\
+ F\Bigl(  z\int_\mathcal{D} e^{\omega_{\mu, L, \varepsilon}(x)}\,dx   \Bigr) \sum\limits_{l=1}^n \sum\limits_{(p_1<\cdots <p_l)} \prod\limits_{i=1}^l \bigl(e^{\mathcal{A}_\varepsilon(x_{p_i})}-1\bigr)
\Bigr] \prod\limits_{i=1}^n e^{ \omega_{\mu,L,\varepsilon}(x_i) },
\end{gather}
where
\begin{equation}\label{intermeq5}
\mathcal{A}_\varepsilon(x) \triangleq
\omega_{\mu-\delta,L,\varepsilon}(x)-\omega_{\mu,L,\varepsilon}(x).
\end{equation}
\end{lemma}
\begin{proof}
The result follows from representing
\begin{equation}
\int_\mathcal{D} e^{\omega_{\mu-\delta, L, \varepsilon}(x)}\,dx = 
\int_\mathcal{D} e^{\omega_{\mu, L, \varepsilon}(s)}\,dx + 
\int_\mathcal{D} e^{\omega_{\mu, L, \varepsilon}(x)}\,\bigl( e^{\mathcal{A}_\varepsilon(x)}-1 \bigr)dx,
\end{equation}
and Taylor expanding in the ``small" parameter \[\int_\mathcal{D} e^{\omega_{\mu, L, \varepsilon}(x)}
\bigl(e^{\mathcal{A}_\varepsilon(x)}-1\bigr) dx\] that vanishes as
$\delta\rightarrow 0.$ \qed
\end{proof}
We can now give a formal derivation of Theorem \ref{theoremdiff} in the form of Eq. \eqref{therulen}.
\begin{proof}
The main idea of the proof is to consider a stochastic flow and derive the corresponding
Feynman-Kac equation regarding intermittency as time.
We can assume $\varphi(x)=1$ without any loss of generality. 
Let 
\begin{equation}
M_{\mu, \varepsilon} = \int_\mathcal{D} e^{\omega_{\mu,
\varepsilon}(x)} dx.
\end{equation}
The starting point is the limit
\begin{equation}\label{thelim}
A\triangleq \frac{\partial}{\partial\delta}\Big\vert_{\delta=0}\,\,
{\bf{E}^*} \left[{\bf{E}} \Bigl[ F\bigl(ze^{B(\delta)} \,M_{\mu, \varepsilon}\bigr)\bigl(ze^{B(\delta)}\bigr)^n\,e^{ \omega_{\mu,\varepsilon}(x_1)+\cdots+\omega_{\mu,\varepsilon}(x_n)}\Bigr]\right],
\end{equation}
where $B(\delta)$ is the Brownian motion with drift $-\delta/2$ independent of
$\omega_{\mu,\varepsilon}(x),$ 
and the star is used to distinguish the expectation with respect to $B(\delta)$ from that with respect to
$\omega_{\mu,\varepsilon}(x).$ By the backward Kolmogorov
equation, we have
\begin{align}
A = &\frac{1}{2} \left[- \frac{\partial}{\partial x}  +\frac{\partial^2}{\partial x^2} \right] \Big\vert_{x=0}\,\,
{\bf{E}} \Bigl[ F\bigl(ze^{x} \,M_{\mu, \varepsilon}\bigr)\bigl(ze^{x}\bigr)^n\,e^{ \omega_{\mu,\varepsilon}(x_1)+\cdots+\omega_{\mu,\varepsilon}(x_n)}\Bigr], \nonumber \\
 = &  \frac{1}{2} z^2\frac{\partial^2}{\partial
z^2}\, {\bf{E}} \Bigl[ F\bigl(z \,M_{\mu, \varepsilon}\bigr)z^n\,e^{ \omega_{\mu,\varepsilon}(x_1)+\cdots+\omega_{\mu,\varepsilon}(x_n)}\Bigr], \nonumber \\
= & \frac{1}{2} z^{n+2} \, {\bf{E}} \Bigl[ F^{(2)}\bigl(z \,M_{\mu, \varepsilon}\bigr) M^2_{\mu, \varepsilon}\,e^{ \omega_{\mu,\varepsilon}(x_1)+\cdots+\omega_{\mu,\varepsilon}(x_n)}\Bigr] + \nonumber \\
 &  + n z^{n+1} \, {\bf{E}} \Bigl[ F^{(1)}\bigl(z \,M_{\mu, \varepsilon}\bigr) M_{\mu, \varepsilon}\,e^{ \omega_{\mu,\varepsilon}(x_1)+\cdots+\omega_{\mu,\varepsilon}(x_n)}\Bigr]+ \nonumber \\
 & + \frac{1}{2} n(n-1)z^n {\bf{E}} \Bigl[ F\bigl(z \,M_{\mu, \varepsilon}\bigr) \,e^{ \omega_{\mu,\varepsilon}(x_1)+\cdots+\omega_{\mu,\varepsilon}(x_n)}\Bigr]
. \label{firsteq}
\end{align}

On the other hand, this limit can be computed in a different way.
By Lemma \ref{keyinvlem}, there holds the following equality in law 
\begin{equation}
F\Bigl(z e^{B(\delta)} M_{\mu, \varepsilon}\Bigr) \prod\limits_{i=1}^n e^{ \omega_{\mu, \varepsilon}(x_i) + B(\delta)} 
=  F\Bigl(z\int_\mathcal{D} e^{\omega_{\mu-\delta,\varepsilon}(x)+{\bar
\omega}_{\delta,e,\varepsilon}(x)} dx\Bigr) 
\prod\limits_{i=1}^n e^{ \omega_{\mu-\delta,\varepsilon}(x_i)+{\bar
\omega}_{\delta,e,\varepsilon}(x_i) }.\label{zidentity}
\end{equation}
Thus, to compute the limit in Eq. \eqref{thelim}, we need to expand
\begin{gather}
{\bf E}^*\left[{\bf E}\Bigl[F\Bigl(z\int_\mathcal{D} e^{\omega_{\mu-\delta,\varepsilon}(x)+{\bar
\omega}_{\delta,e,\varepsilon}(x)} dx\Bigr)  \prod\limits_{i=1}^n e^{ \omega_{\mu-\delta,\varepsilon}(x_i)+{\bar
\omega}_{\delta,e,\varepsilon}(x_i) }\Bigr]\right]  \label{intermeq4} 
\end{gather}
in $\delta$ up to $o(\delta)$ terms. The star now indicates the expectation with respect to
${\bar
\omega}_{\delta,e,\varepsilon}(x),$ which is independent of $\omega_{\mu-\delta,\varepsilon}(x)$ by construction. 
Let $\mathcal{A}_\varepsilon(x)
\triangleq
\omega_{\mu-\delta,\varepsilon}(x)-\omega_{\mu,\varepsilon}(x)$
as in Eq. \eqref{intermeq5} and
\begin{gather}
\bar{\mathcal{A}}_\varepsilon(x)\triangleq  {\bar
\omega}_{\delta,e,\varepsilon}(x), \\
\mathcal{C}\triangleq \int_\mathcal{D} e^{\omega_{\mu,\varepsilon}(x)}
\bigl(e^{\mathcal{A}_\varepsilon(x)+\bar{\mathcal{A}}_\varepsilon(x)}-1\bigr)\,dx.
\end{gather}
While we do not know how to expand these quantities 
in $\delta,$ they all clearly vanish as $\delta\rightarrow 0$ so 
we can write
\begin{gather}
F\Bigl(z\int_\mathcal{D} e^{\omega_{\mu-\delta,\varepsilon}(x)+{\bar
\omega}_{\delta,e,\varepsilon}(x)} dx\Bigr) = \sum\limits_{l=0}^\infty \frac{z^l}{l!} F^{(l)}(z  M_{\mu, \varepsilon}) \,\mathcal{C}^l, \\
\prod\limits_{i=1}^n e^{ \omega_{\mu-\delta,\varepsilon}(x_i)+{\bar
\omega}_{\delta,e,\varepsilon}(x_i) } = \prod\limits_{i=1}^n e^{ \omega_{\mu,\varepsilon}(x_i) } 
\prod\limits_{i=1}^n \Bigl(1+ \bigl(e^{ \mathcal{A}_\varepsilon(x_i) + \bar{\mathcal{A}}_\varepsilon(x_i) }-1\bigr)\Bigr).
\end{gather}
It
follows that the expression in Eq. \eqref{intermeq4} can be written as
\begin{gather}
{\bf E}^*\left[{\bf E}\Bigl[ \sum\limits_{l=0}^\infty \frac{z^l}{l!} F^{(l)}(z  M_{\mu, \varepsilon}) \,\mathcal{C}^l \prod\limits_{i=1}^n 
e^{ \omega_{\mu,\varepsilon}(x_i) } 
\prod\limits_{i=1}^n \Bigl(1+ \bigl(e^{ \mathcal{A}_\varepsilon(x_i) + \bar{\mathcal{A}}_\varepsilon(x_i) }-1\bigr)\Bigr) \Bigr]\right]. \label{intermeq3}
\end{gather}
The advantage of this representation is that the only
$\bar{\omega}_\varepsilon$ dependence is in
$\bar{\mathcal{A}}_\varepsilon(s).$ This allows us to compute the
$\bf{E}^*$ expectation in Eq. \eqref{intermeq4}. Indeed, Eq. \eqref{intermeq4} entails two
expectations: the $\bf{E}$ with respect to $\omega_\varepsilon$
process 
and the $\bf{E}^*$ expectation with respect to
$\bar{\omega}_\varepsilon$ process. Interchanging their order, it
follows from Eq. \eqref{intermeq3} that computing the $\bf{E}^*$ expectation is
now reduced to computing 
\begin{equation}
{\bf E}^* \Bigl[ \mathcal{C}^l \, \prod\limits_{j=1}^{k} 
\bigl(e^{ \mathcal{A}_\varepsilon(x_{p_j}) + \bar{\mathcal{A}}_\varepsilon(x_{p_j}) }-1\bigr)\Bigr],
\end{equation}
where $(p_1<\cdots <p_k)$ denotes an increasing tuple of numbers from $1\cdots n$ of length $k,$ due to the algebraic
identity in Eq. \eqref{auxalg}.
As
$\mathcal{A}_\varepsilon(x)$ and $\bar{\mathcal{A}}_\varepsilon(x)$
are independent processes, it follows from Lemma \ref{lemmaexpec} applied to
$\mathcal{B}(s)=\exp\bigl(\mathcal{A}_\varepsilon(x)+\bar{\mathcal{A}}_\varepsilon(x)\bigr)-1$
with $\mathfrak{f}(\delta, x)\triangleq \mathcal{A}_\varepsilon(x)$
that we have the estimates 
\begin{gather}
{\bf E}^*[\mathcal{B}(x_1)\mathcal{B}(x_2)]=
\bigl(e^{\mathcal{A}_\varepsilon(x_1)}-1\bigr)\bigl(e^{\mathcal{A}_\varepsilon(x_2)}-1\bigr)+
\delta \,g_{e,\varepsilon}(x_1,x_2)+o(\delta), \\
{\bf E}^*[\mathcal{B}(x_1)\cdots\mathcal{B}(x_k)]=
\bigl(e^{\mathcal{A}_\varepsilon(x_1)}-1\bigr)\cdots\bigl(e^{\mathcal{A}_\varepsilon(x_k)}-1\bigr)+o(\delta),\,\,k\neq
2.
\end{gather}
Collecting what we have shown so far and using Lemma \ref{lemmadiff}, we obtain
\begin{align}
A= & -\frac{\partial}{\partial\mu} \Bigl[ F\bigl(  z M_{\mu, \varepsilon}  \bigr)  \prod\limits_{i=1}^n e^{ \omega_{\mu,\varepsilon}(x_i) } \Bigr] +\nonumber \\
& + \frac{z^{n+2}}{2}\int_{\mathcal{D}^2}{\bf E}\Bigl[ F^{(2)} (zM_{\mu, \varepsilon})  
e^{\omega_{\mu,\varepsilon}(y_1)+\omega_{\mu,\varepsilon}(y_2)}\prod\limits_{i=1}^n e^{ \omega_{\mu,\varepsilon}(x_i) } \Bigr]
\,g_{e,\varepsilon}(y_1,y_2)dy_1dy_2 +\nonumber\\
& + z^{n+1}\int_{\mathcal{D}}{\bf E}\Bigl[ F^{(1)} (zM_{\mu, \varepsilon})  e^{\omega_{\mu,\varepsilon}(y)}\prod\limits_{i=1}^n e^{ \omega_{\mu,\varepsilon}(x_i) } 
\Bigr]
\,\sum\limits_{k=1}^n g_{e,\varepsilon}(y,x_k)dy    +\nonumber\\
& + z^n\,{\bf E}\Bigl[ F(zM_{\mu, \varepsilon})  \prod\limits_{i=1}^n e^{ \omega_{\mu,\varepsilon}(x_i) } 
\Bigr]
\,\sum\limits_{l<k}^n g_{e,\varepsilon}(x_l,x_k).\label{intermeq1}
\end{align}
Finally, observing that
$g_{e,\varepsilon}(x_1,x_2)=1+g_{1,\varepsilon}(x_1,\,x_2)$ and
comparing the resulting expression for $A$ with that in Eq. \eqref{firsteq},
and then letting $z=1$ and $\varepsilon\rightarrow 0,$ we arrive at Eq. \eqref{therulen}. \qed
\end{proof}


\end{document}